\documentclass{article}

    \usepackage[preprint]{neurips_2022}

\usepackage[utf8]{inputenc} 
\usepackage[T1]{fontenc}    
\usepackage{hyperref}       
\usepackage{url}            
\usepackage{booktabs}       
\usepackage{amsfonts}       
\usepackage{amssymb}
\usepackage{nicefrac}       
\usepackage{microtype}      
\usepackage{xcolor}         

\usepackage{algorithm,algorithmic}

\usepackage{wrapfig}
\usepackage{amsthm}
\usepackage{mathtools} 
\usepackage{enumitem}
\usepackage{xcolor}
\usepackage{nicefrac}       
\definecolor{mydarkgreen}{RGB}{39,130,67}
\definecolor{mydarkred}{RGB}{192,47,25}

\usepackage{url}  

\usepackage[para,online,flushleft]{threeparttable}
\usepackage{tcolorbox}
\usepackage{pifont}
\usepackage{cutwin}

\def \O{\mathcal O}

\def\D{\mathcal D}
\def\X{\mathcal X}
\def\Y{\mathcal Y}

\def\R{\mathbb R}
\def\E{\mathbb E}
\def\Prob{\mathbb P}

\def\Z{\mathcal{Z}}
\def\e{\epsilon}
\def\la{\langle}
\def\ra{\rangle}
\def\vp{\varphi}

\def\be{\boldsymbol{e}}

\def\lm{\lambda}

\newcommand{\cmark}{{\ding{51}}}
\newcommand{\xmark}{{\ding{55}}}

\newtheorem{theorem}{Theorem}
\newtheorem{lemma}{Lemma}

 \newtheorem{example}{Example}
\newtheorem{remark}{Remark}
\newtheorem{assumption}{Assumption}

\def\dd#1{{\color{blue}#1}} 

\title{Gradient-Free  Optimization for Non-Smooth Saddle Point Problems under Adversarial Noise}

\author{
 Darina   Dvinskikh\\
  MIPT, \\
    IITP  RAS,\\
  ISP RAS Research Center for Trusted Artificial Intelligence,  \\
  \texttt{dviny.d@yandex.ru}
  \And
  Vladislav Tominin\\
 MIPT\\
  \texttt{tominin.vd@phystech.edu}
  \And
    Yaroslav  Tominin\\
    MIPT\\
  \texttt{tominin.yad@phystech.edu}
  \And
  Alexander  Gasnikov\\
  MIPT,\\
  IITP  RAS,\\
  ISP RAS Research Center for Trusted Artificial Intelligence\\
  \texttt{gasnikov@yandex.ru}
}

\begin{document}

\maketitle

\begin{abstract}
We consider non-smooth saddle point optimization problems. To solve these problems, we propose a zeroth-order method under bounded or Lipschitz continuous noise, possible adversarial. In contrast to the state-of-the-art algorithms, our algorithm is optimal in terms of both criteria: oracle calls complexity and the maximum value of admissible noise. The proposed method is simple and easy to implement as it is built  on zeroth-order version of the stochastic mirror descent. The convergence analysis is given in terms of the average and probability. We also pay special attention to the duality gap $r$-growth condition $(r\geq 1)$, for which we provide a modification of our algorithm using the restart technique. We also comment on infinite noise variance and upper bounds in the case of Lipschitz noise. The results obtained in this paper are significant not only for saddle point problems but also for convex optimization.
\end{abstract}

\section{Introduction}\label{sec:intro}
In this paper, we consider  stochastic non-smooth \textit{saddle point} problems of the following form
\begin{equation}\label{eq:min_max_problem}
    \min_{x\in \X} \max_{y\in \Y} ~f(x,y),
\end{equation}
where $f(x,y) \triangleq \E_\xi \left[f(x,y,\xi)\right]$ is the expectation, w.r.t. $\xi \in \Xi$,  $f: \X \times\Y \to \R $   is convex-concave and Lipschitz continuous, and  $\X \subseteq \R^{d_x}$, $\Y \subseteq \R^{d_y}$ are  convex compact sets. The standard interpretation of such min-max problems is the antagonistic game between a learner and an adversary, where the equilibria are the saddle points  \cite{neumann1928theorie}.
Now the interest in   saddle point problems 
is renewed due to the popularity of generative adversarial networks (GANs), whose training involves solving  min-max problems \cite{goodfellow2014generative,chen2017zoo}. 

Motivated by  many applications in the field of  reinforcement learning \cite{choromanski2018structured,mania2018simple} and statistics,  where only a  black-box access to  objective values is available, we consider \textit{zeroth-order} oracle (also known as \textit{gradient-free} oracle).  Particularly, we mention the classical problem of adversarial multi-armed bandit \cite{flaxman2004online,bartlett2008high,bubeck2012regret}, where a learner receives a feedback  given by the function evaluations from an adversary. Thus, zeroth-order methods  \cite{conn2009introduction}  are the workhorse technique when the    gradient information is prohibitively expensive or even  not available and optimization is performed based only on the function evaluations.

\paragraph{Related Work.}

Zeroth-order methods in the non-smooth setup were developed in a wide range of works  \cite{polyak1987introduction,spall2003introduction,conn2009introduction,duchi2015optimal,shamir2017optimal,nesterov2017random,gasnikov2017stochastic,beznosikov2020gradient,gasnikov2022power}.  Particularly, in \cite{shamir2017optimal}, an optimal algorithm  was provided as an improvement to the work \cite{duchi2015optimal} for a non-smooth case but Lipschitz continuous  in stochastic convex optimization problems.
However, this algorithm uses the exact function evaluations that can be infeasible in some applications. 
Indeed,  objective $f(z,\xi)$    can be not directly observed but instead, its noisy approximation $\vp(z,\xi) \triangleq f(z,\xi) + \delta(z)$ can be queried, where $\delta(z)$ is some  noise, possibly adversarial.
This noisy-corrupted setup was considered in
many works \dd{\cite{polyak1987introduction,granichin2003randomized}}, however, such an algorithm that is optimal in terms of the number of oracle calls complexity and the maximum value of the noise has not been proposed.
For instance,   in \cite{bayandina2018gradient,beznosikov2020gradient},  optimal algorithms in terms of oracle calls complexity were proposed, however, they are not optimal in terms of the maximum value of the noise. 
 In papers \cite{risteski2016algorithms,vasin2021stopping},   algorithms are optimal in terms of the maximum value of the noise, however, they are not optimal in terms of the oracle calls complexity.  
This paper presents a new algorithm which is optimal both in terms of the inexact  oracle calls complexity and the maximum value of admissible noise. The method is built on a gradient-free version of the mirror descent  with an inexact oracle.  We consider two possible scenarios for the nature of the  noise arising in different applications: the noise is bounded  or  is Lipschitz continuous.  Table \ref{Table1} demonstrates our contribution by comparing our results with the existing optimal bounds, where $\e$ is the desired accuracy to solve problem \eqref{eq:min_max_problem} and $d$ is the problem dimension.
We notice that the results obtained for saddle point problems are also valid for convex optimization. 

\begin{table}[!h]
   \centering
    \small
    \caption{Summary of the Contribution}
    \label{Table1}
    \begin{threeparttable}
    \begin{tabular}{llllll}
      \toprule
         \scriptsize \textsc{Paper} 
        & \begin{tabular}{l} 
          \scriptsize \textsc{Problem}
        \end{tabular} 
        & \begin{tabular}{l}
      \scriptsize \textsc{Expectation}  \\ 
        \scriptsize \textsc{or Large } \\
          \scriptsize \textsc{ Deviation }
        \end{tabular}  
        & \begin{tabular}{l}
         \scriptsize \textsc{Is } \\ 
      \scriptsize \textsc{the  Noise} \\ 
      \scriptsize \textsc{Lipschitz?}
        \end{tabular} 
         & \begin{tabular}{l}
         \scriptsize \textsc{Number of} \\  \scriptsize \textsc{Oracle Calls} \\ 
        \end{tabular} 
            & \begin{tabular}{l}
      \scriptsize \textsc{Maximum Value}  \\  \scriptsize \textsc{of the Noise} 
        \end{tabular} \\
     \midrule
\tiny{\citet{bayandina2018gradient}}   &   \scriptsize{ convex}  & $\E$& \xmark   &   $\nicefrac{d}{\e^2}$ & $\nicefrac{\e^2}{d^{3/2}}$\\   
    \tiny{\citet{beznosikov2020gradient}}   &   \scriptsize{ saddle point} & $\E$ & \xmark  &  $\nicefrac{d}{\e^2}$ & $\nicefrac{\e^2}{d}$\\ 
    \tiny{\citet{vasin2021stopping}}   &   \scriptsize{ convex} & $\E$ & \xmark   &  ${\rm Poly}\left(d, \nicefrac{1}{\e}\right)$ & $\nicefrac{\e^2}{\sqrt d}$\\\ 
 
\tiny{\citet{risteski2016algorithms}}   &   \scriptsize{ convex} & $\E$ & \xmark   & $ {\rm Poly}\left(d, \nicefrac{1}{\e}\right)$ & $\max\left\{ \nicefrac{\e^{2}}{\sqrt{d}}, \nicefrac{\e}{d}\right\}$\tnote{{\color{blue}(1)}}~ ~\\ 
    { \scriptsize \textsc{This work}}     &  \scriptsize{  saddle point}  &  $\E$~ and ~$\Prob$ & \xmark   & $\nicefrac{d}{\e^2}$ & $\nicefrac{\e^2}{\sqrt d}$\\    
    
       { \scriptsize \textsc{This work}}     &  \scriptsize{  saddle point}   & $\E$~ and ~$\Prob$ & \cmark  & $\nicefrac{d}{\e^2}$ & $\nicefrac{\e}{\sqrt d} $~ \tnote{{\color{blue}(2)}} \\    
 \bottomrule
    \end{tabular} 
    
\begin{tablenotes}
 {\scriptsize
 \item[{\color{blue}(1)}] This bound  is also the upper bound up to a logarithmic factor. 
 In the large-scale setup ($\e^{-2}\lesssim d$), the maximum is reached on the second term, namely $\e^2/\sqrt{d}$.
}
\\
{\scriptsize
 \item[{\color{blue}(2)}] All of the estimates, except this one, in this column are for the maximum value of the noise. This estimate is the  estimate of the Lipschitz constant as now  the noise is Lipschitz continuous. 
}
\end{tablenotes}    
\end{threeparttable}

\end{table}



\paragraph{Contribution.}
Now we list our contribution as follows
\begin{itemize}[leftmargin=*]
    \item We provide an algorithm which is optimal in terms of number of oracle calls and maximum value of addmisible noise. We state the results about its convergence in expectation and probability 
    \item For the $r$-growth condition, we restate the  results for the proposed algorithm run with restarts
    \item We comment on how the results can be modified under infinite noise variance
    \item We comment on  `upper' bound in the case of Lipschitz noise
\end{itemize}

\paragraph{Paper Organization.} This paper is organized as follows.  In Section \ref{sec:main_alg}, we present  the main algorithm of the paper and  analysis of its convergence. 
In Section \ref{sec:rest}, under additional assumption of $r$-growth condition we restate the results for the proposed algorithm run with restarts. In Section 
\ref{sec:inf_noise}, we comment on the case of infinite noise variance. 
Finally, Section \ref{sec:upper_bound} gives some ideas about upper bounds in the case of Lipschitz noise.


\section{Zeroth-order algorithm}\label{sec:main_alg}
In this section,  we present an
algorithm (see Algorithm \ref{alg:zoSPA}) that is optimal in terms of the number of inexact zeroth-order oracle calls and the maximum value of adversarial noise.  The algorithm  is based on a gradient-free version of the stochastic  mirror descent (SMD) \cite{ben2001lectures}.  We start with some key notation, background material and assumptions.

\subsection{Notation and assumptions} 
We use $\la x,y\ra \triangleq \sum_{i=1}^d x_i y_i$ to define the inner product of $x,y \in \R^d$, where $x_i$ is
  the $i$-th component of $x$. By  norm  $\|\cdot\|_p$ we mean  the $\ell_p$-norm. Then the dual norm of the norm $\|\cdot\|_p$ is $\|\lm\|_q \triangleq \max\{\la x,\lm \ra \mid \|x\|_p \leq 1 \} $. Operator $\E[\cdot]$  is the full  expectation and operator $\E_{\xi}[\cdot]$ is the conditional expectation, w.r.t. $\xi$.   Let us introduce the embedding space $\Z \triangleq \X \times \Y$, and then  some $z\in \Z$ means $z \triangleq (x,y)$, where $x\in \X, y \in \Y$.
On this embedding space, we introduce  the $\ell_p$-norm
and  a prox-function $\omega(z)$ compatible with this norm. Then we define the  Bregman divergence associated with $\omega(z)$ as
\[V_z( v) \triangleq \omega(z) -\omega(v) - \la  \nabla \omega(v), z - v \ra \geq \|z-v\|_p^2/2, \quad \mbox{ for all } z,v \in Z.\] 
We also introduce a prox-operator as follows
\[ {\rm Prox}_z(\xi) \triangleq \arg\min_{v\in Z} \left(V_z(v) + \la \xi, v \ra\right), \quad \mbox{ for all } z \in Z.\]
Finally, we denote the $\omega$-diameter  of $\Z$ by 
$\D \triangleq \max\limits_{z,v\in \Z}\sqrt{2V_{z}(v)}= \widetilde{\O}\left(\max\limits_{z,v\in \Z}\|z-v\|_p\right)$. Here $\widetilde \O\left(\cdot\right)$ is $\O\left(\cdot\right)$ up to a $\sqrt{\log d}$-factor.

\begin{assumption}[Lischitz continuity of the objective]\label{ass:Lip_obj}
 Function $f(z,\xi)$ is $M_2(\xi)$-Lipschitz continuous in $z\in \Z$ w.r.t. the $\ell_2$-norm, i.e.,   for all  $z_1,z_2\in \Z$ and $\xi \in \Xi$, 
 \[|f(z_1, \xi)- f(z_2,\xi)| \leq M_2(\xi) \|z_1- z_2\|_2.\] 
 Moreover, there exists a positive constant $ M_2$ such that  $\E\left[M^2_2(\xi)\right] \leq M_2^2$.
\end{assumption}

\begin{assumption}[Boundedness of the noise]\label{ass:err_noise}
For all  $z \in \Z$, it holds $|\delta(z)| \leq \Delta.$
\end{assumption}

\begin{assumption}[Lipschitz continuity of the noise]\label{ass:Lip_noise}
Function $\delta(z)$ is $M_{2,\delta}$-Lipschitz continuous in $z\in \Z$ w.r.t. the $\ell_2$-norm, i.e., for all $z_1, z_2 \in \Z$, \[|\delta(z_1)- \delta(z_2)| \leq M_{2,\delta} \|z_1- z_2\|_2.\]
\end{assumption}

\subsection{Black-box oracle and gradient approximation.} We assume that we can query   zeroth-order oracle corrupted by an adversarial noise $\delta(z)$: 
\begin{equation}\label{eq:zerothorder}
    \vp(z,\xi) \triangleq f(z,\xi) + \delta(z).
\end{equation}
 The gradient of $\vp(z,\xi)$ from \eqref{eq:zerothorder}, w.r.t. $z$, can be approximated by the function evaluations in two random points closed to $z$. To do so, we define  vector $\be$ picked uniformly at random from the Euclidean unit sphere $\{\be:\|\be\|_2= 1\}$. Let $\be \triangleq (\be_x^\top, -\be_y^\top)^\top$, where ${\rm dim}(\be_x) \triangleq d_x$, ${\rm dim}(\be_y) \triangleq d_y$ and ${\rm dim}(\be) \triangleq d = d_x+d_y$. Then the gradient  of $ \vp(z,\xi)$ can be estimated by the following approximation with a  small variance \cite{shamir2017optimal}:
\begin{equation}\label{eq:grad_approx}
    g(z,\xi, \be) = \frac{d}{2\tau}\left(\vp(z+\tau \be,\xi) - \vp(z-\tau \be,\xi)   \right) \begin{pmatrix} \be_x \\ -\be_y \end{pmatrix},
\end{equation}
where $\tau>0$ is some constant. 

\subsection{Randomized smoothing.}
  Unfortunately, this standard zeroth-order approximation \eqref{eq:grad_approx} is a poor estimator for subgradients of
 a non-smooth objective. To support this, let us consider the following example.

    \begin{wrapfigure}{R}{0.38\textwidth}
    \begin{minipage}{0.38\textwidth}
\vspace{-0.3cm}
 \hspace{0.3cm}  \includegraphics[width=0.9\linewidth]{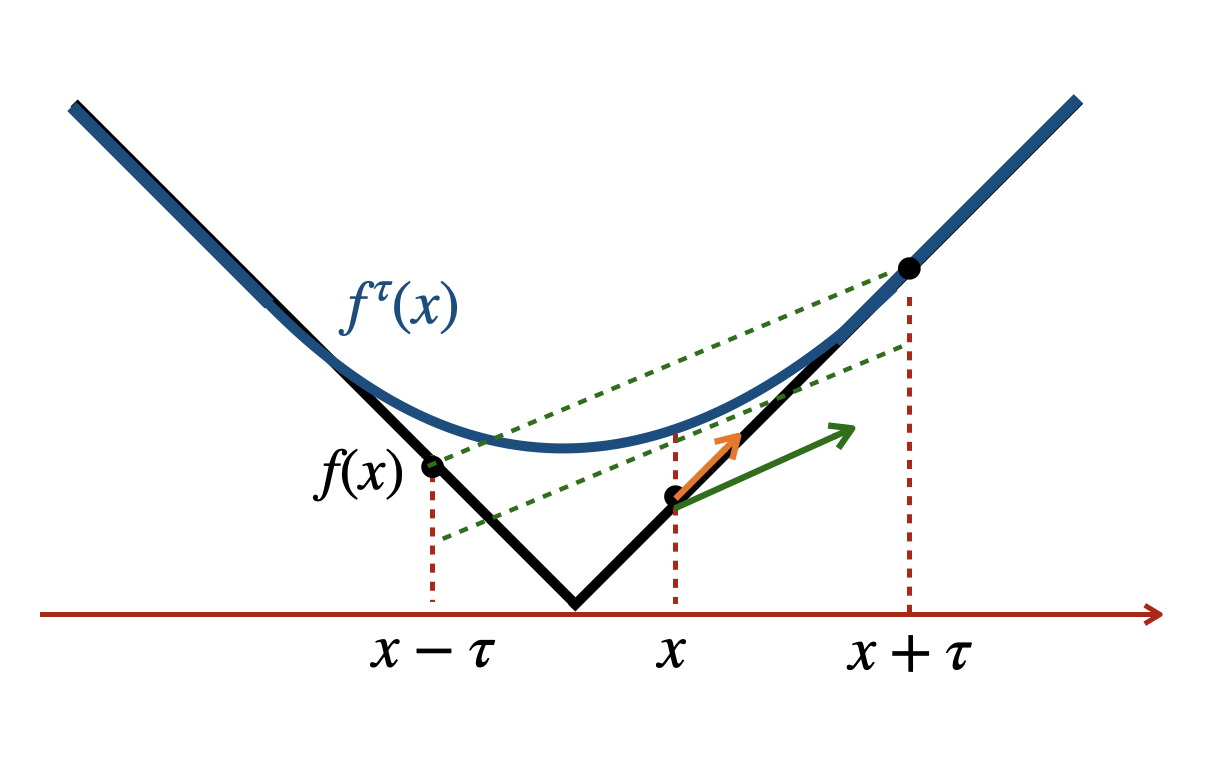}
   \vspace{-0.6cm}
 \caption{Smooth approximation $f^{\tau}(x)$ of a non-smooth function $f(x)$}
 \label{fig:non_smooth_funct}
    \end{minipage}
       \vspace{-0.2cm}
      \end{wrapfigure}

\begin{example}[one-dimensional case]
$f(x) = |x|$.  Zeroth-order approximation \eqref{eq:grad_approx} of subgradients of $f(x)$ can be simplified as
\begin{equation}\label{eq:simpgrad}
    g(x) = \frac{1}{2\tau}(f(x+\tau) - f(x-\tau)).
\end{equation}
For point $x \in [-\tau,\tau]$,
$g(x) =\pm  {x}/{\tau}$.
 However, for all $x>0$, $\nabla f(x)=1$ and for all  $x<0$, $\nabla f(x)=-1$.
\end{example}

Since  the problem \eqref{eq:min_max_problem} is non-smooth, we  
introduce the following smooth approximation of a non-smooth function  (see Figure \ref{fig:non_smooth_funct} as an example) 
\begin{equation}\label{eq:smoothedfunc}
    f^\tau(z) \triangleq \E_{\tilde \be} f(z+\tau \tilde \be),
\end{equation}
where $\tau>0$ and $ \tilde\be$ is a vector picked uniformly at random from the Euclidean unit ball: $\{\tilde\be:\|\tilde\be\|_2\leq 1\}$. Function  $f^\tau(z)$ can be referred as a smooth approximation of $f(z)$ and it will be used only for  deriving the convergence rate of proposed algorithm. Here $f(z)\triangleq \E f(z,\xi)$.  

The next lemma presents the quality of such an approximation.
\begin{lemma}\label{lm:shamirSmoothFun}
Let $f(z)$ be $M_2$-Lipschitz continuous function. Then for $f^\tau(z)$ from \eqref{eq:smoothedfunc}, it holds
\[
\sup_{z\in \Z}|f^\tau(z) - f(z)|\leq \tau M_2.
\]
\end{lemma}
\begin{proof}
By the definition of $f^\tau(z)$ from \eqref{eq:smoothedfunc} and Assumption \ref{ass:Lip_obj} we have
\begin{align*}
    |f^\tau(z) - f(z)|
    &= \left|\E_{\be}\left[f(z+\tau\be)\right]  -f(z)\right| \leq \E_{\be}\left[|f(z+\tau\be) - f(z)| \right] \leq \E\left[ M_2 \|\tau \be\|_2\right] =  M_2 \tau.  
\end{align*}
\end{proof}
\begin{lemma}\label{lm:flax}
 Function  $f^\tau(z)$ is differentiable with the following gradient
\[
\nabla  f^\tau(z) = \E_{{\be}}\left[ \frac{d}{\tau} f(z+\tau {\be})\be \right].
\]
\end{lemma}


  \begin{wrapfigure}{R}{0.45\textwidth}
    \begin{minipage}{0.45\textwidth}
     \vspace{-0.8cm}
\begin{algorithm}[H]
 \caption{Zeroth-order  \textsc{SMD}}
\label{alg:zoSPA}
\begin{algorithmic}[1]
 \REQUIRE iteration number $N$,\\
$z^1  \leftarrow \arg\min_{z\in \Z} \omega(z)$
\FOR{$k=1,\dots, N$}{
  \STATE Sample $\be^k, \xi^k$ independently
 \STATE Initialize $\gamma_k 
\to \frac{\D}{M}\sqrt{\frac{2}{N}}$ with   $M$ defined by \eqref{eq33} or \eqref{eq44}
 \STATE
 Calculate $g(z^k, \xi^k, \be^k)$ via \eqref{eq:grad_approx}
 \STATE
  $z^{k+1}  \leftarrow {\rm Prox}_{z^k}\left(\gamma_k g(z^k, \xi^k, \be^k)\right)$
}
\ENDFOR
        \ENSURE    $\hat z^N  \leftarrow \left(\sum_{i=1}^{N}\gamma_k\right)^{-1}\sum_{k=1}^{N} \gamma_k z^k$
\end{algorithmic}
\end{algorithm}
\vspace{-0.6cm}
    \end{minipage}
  \end{wrapfigure}
  \subsection{Algorithm and its convergence rate}
  Now we present zeroth-order
algorithm to solve problem \eqref{eq:min_max_problem} (see Algorithm \ref{alg:zoSPA}).   The stepsize $\gamma_k 
= \frac{\D}{M}\sqrt{\frac{2}{N}}$, where positive constant $M$   is chosen as: 
\begin{enumerate}
    \item under Assumption \ref{ass:err_noise} 
    \begin{equation}\label{eq33}
        M^2\triangleq \O\left(da_q^2M_2^2 + {d^2a_q^2\Delta^2}{\tau^{-2}}\right)
    \end{equation}
    \item under Assumption \ref{ass:Lip_noise} \begin{equation}\label{eq44}
        M^2 \triangleq \O\left( d a_q^2\left(M_2^2+M_{2,\delta}^2\right)\right),
    \end{equation}
\end{enumerate}
 where $N$ is the number of algorithm iterations and $a_q^2 \triangleq \O\left(\sqrt{\E\left[ \|\be\|_q^4\right]}\right)= \O\left( \min\{q,\log d \} d^{2/q-1} \right)  $ \cite{gorbunov2019upper}.
  
The next theorem presents the convergence rate of the  Algorithm \ref{alg:zoSPA} in terms of the expectation.
\begin{theorem}\label{cor:bound_noise}
Let  $\e$ be the desired accuracy to solve problem \eqref{eq:min_max_problem} and  $\tau$  from randomized smoothing \eqref{eq:smoothedfunc} be chosen as $\tau = \O\left(\e/M_2\right)$. Let function $f(x,y,\xi)$  satisfy the Assumption \ref{ass:Lip_obj} and one of the two following statement is true
\begin{enumerate}
    \item Assumption \ref{ass:err_noise} holds with $\Delta =\O\left( \frac{\e^2}{\D M_2\sqrt{ d}}\right)$.
    \item Assumption \ref{ass:Lip_noise} holds with
 $M_{2,\delta} =\O\left( \frac{\e}{\D \sqrt{ d}}\right)$. 
\end{enumerate}
Then
for  $ \e_{\rm sad}  \triangleq \max_{y\in \Y}  f(\hat x^N, y) - \min_{x\in \X}   f(x, \hat y^N)$ with  the output of   Algorithm \ref{alg:zoSPA} $ \hat z^N \triangleq (\hat x^N, \hat y^N)$ , it holds   $\E\left[\e_{\rm sad}(\hat z^N)\right]\leq \e$  after the following number of iterations
\[N =  \O\left( {dM_2^2\D^2 a_q^2}/{\e^{2}} \right).\]
\end{theorem}
Next we specify the Theorem \ref{cor:bound_noise} in the two following special setups: the $\ell_2$-norm and the $\ell_1$-norm in the two following examples.
\begin{example}[$\ell_2$-norm]
Let $p=2$, then $q=2$ and  $\sqrt{\E \left[ \|\be\|_2^4\right]}   = 1$. Thus, $a_2^2=1$ and $\D^2 = \max\limits_{z,v\in \Z}\|z-v\|^2_2$. Consequently, the number of iterations in the Corollary \ref{cor:bound_noise} can be rewritten as follows
\[N =  \O\left( \frac{dM_2^2}{\e^{2}}\max\limits_{z,v\in \Z}\|z-v\|^2_2\right).\]
\end{example}
\begin{example}[$\ell_1$-norm]\citet[Lemma 4]{shamir2017optimal}
Let $p=1$ then, $q=\infty$ and $\sqrt{\E \left[ \|\be\|_\infty^4\right]}   = \O\left({\frac{\log d}{d}}\right)$. Thus, $a_\infty^2=\O\left({\frac{\log d}{d}}\right)$ and $\D^2 = \O\left(\log d \max\limits_{z,v\in \Z}\|z-v\|^2_1\right)$. Consequently, the number of iterations in the Corollary \ref{cor:bound_noise} can be rewritten as follows
\[N =  \O\left( \frac{(\log d)^2 M_2^2}{\e^{2}}\max\limits_{z,v\in \Z}\|z-v\|_1^2 \right).\]
\end{example}


\begin{remark}[Variable separation]
Heretofore we assumed that 
the  proximal setups for spaces $\X$
and $\Y$ are the same. In some applications, this is not the case.  For instance, when spaces $\X$ and $\Y$ require different Bregman divergences. In this case, we  can   replace the proximal step  $z^{k+1}  \leftarrow {\rm Prox}_{z^k}\left(\gamma_k g(z^k, \xi^k, \be^k)\right)$ in Algorithm \ref{alg:zoSPA}  by two proximal steps on spaces $\X$ and $\Y$ respectively
\begin{equation*}
    x^{k+1}  \leftarrow {\rm Prox}_{x^k}\left(\gamma_k g_x(z^k, \xi^k, \be^k)\right); \qquad  y^{k+1}  \leftarrow {\rm Prox}_{y^k}\left(\gamma_k g_y(z^k, \xi^k, \be^k)\right).
\end{equation*}
\end{remark}

\begin{proof}[Proof of Theorem
\ref{cor:bound_noise}] 
For brevity, we provide the proof only under Assumption \ref{ass:err_noise}. The convergence rate under Assumption \ref{ass:Lip_noise} can be obtained similarly. By the definition $z^{k+1}= {\rm Prox}_{z^k}\left(\gamma_k g(z^k,\be^k,\xi^k) \right)$ we get  \cite{ben2001lectures}, for all $u\in \Z$
\begin{align*}
   \gamma_k \la g(z^k,\be^k,\xi^k), z^k-u\ra \leq V_{z^k}(u) - V_{z^{k+1}}(u) + \gamma_k^2\|g(z^k,\be^k,\xi^k)\|_q^2/2.   
\end{align*}
Taking the conditional  expectation w.r.t. $\xi, \be$ and summing for $k=1,\dots,N$  we obtain, for all $u\in \Z$
\begin{align}\label{eq:inmainsqE0}
 \sum_{k=1}^N \gamma_k  \E_{\be^k,\xi^k} \left[ \la g(z^k,\be^k,\xi^k), z^k-u\ra\right] 
   &\leq  V_{z^1}(u)  + \sum_{k=1}^N\frac{\gamma_k^2}{2} \E_{\be^k,\xi^k}\left[\|g(z^k,\be^k,\xi^k)\|_q^2 \right]. 
\end{align}
\begin{lemma}\label{lm:shamirEnabla}
 For $g(z,\xi,\be)$ from \eqref{eq:grad_approx}, the following holds
 under Assumption \ref{ass:Lip_obj} for $c>0$ 
 \begin{enumerate}
     \item  and Assumption  \ref{ass:err_noise}
     $\E_{\xi,\be}\left[\|g(z,\xi,\be)\|^2_q\right] \leq c  dM_2^2 a^2_q + {d^2 \Delta^2}{\tau^{-2}}a_q^2,$
     \item  and Assumption \ref{ass:Lip_noise} $\E_{\xi,\be}\left[\|g(z,\xi,\be)\|^2_q\right] \leq c  d (M_2^2+M_{2,\delta}^2)a^2_q.$
 \end{enumerate}
\end{lemma}

{\bf Step 1.} For the second term in the r.h.s of  \eqref{eq:inmainsqE0} we use Lemma \ref{lm:shamirEnabla} and get
 under Assumption  \ref{ass:err_noise}:
     \begin{equation}\label{eq:inmainsqEtwo10}
        \E_{\be^k,\xi^k}\left[\|g(z^k,\xi^k,\be^k)\|^2_q\right] \leq c  dM_2^2 a^2_q + {d^2 \Delta^2}{\tau^{-2}}a_q^2, 
     \end{equation}
 where    $c$ is some numerical constant and $\sqrt{\E \left[ \|\be^k\|_q^4\right]} \leq a_q^2$.

\begin{lemma}\label{lm:expecapproxgrad}
   For $g(z,\xi,\be)$ from \eqref{eq:grad_approx} and $f^\tau(z)$ from \eqref{eq:smoothedfunc}, the following holds 
 \begin{enumerate}
 \item under Assumption \ref{ass:err_noise}
$\E_{\xi, \be}\left[\la g(z,\xi,\be),r\ra\right] \geq \la \nabla f^\tau(z),r\ra  - {d\Delta}{\tau^{-1}} \E_{\be} \left[\left|\la \be, r \ra \right|\right],$
     \item  under Assumption \ref{ass:Lip_noise}   $\E_{\xi, \be}\left[\la g(z,\xi,\be),r\ra\right] \geq \la \nabla f^\tau(z),r\ra   - d M_{2,\delta} \E_{\be} \left[\left|\la  \be, r \ra \right|\right],$
 \end{enumerate}
 \end{lemma}

{\bf Step 2.} For the l.h.s. of   \eqref{eq:inmainsqE0} and  $u \triangleq (x^\top,y^\top)^\top$, we use Lemma \ref{lm:expecapproxgrad}
under Assumption \ref{ass:err_noise}
 \begin{align}\label{eq:eqrefeqffxyz110}
     \sum_{k=1}^N \gamma_k \E_{\be^k,\xi^k}\left[\la g(z^k,\be^k,\xi^k), z^k-u\ra \right]
     &\geq  \sum_{k=1}^N \gamma_k \la \nabla f^\tau(z^k), z^k-u \ra  \notag \\
     &-\sum_{k=1}^N \gamma_k \E_{\be^k}\left[ \left|\left\la  {d\Delta}{\tau^{-1}}  \be^k, z^k-u  \right\ra \right|\right].
\end{align}
For the first term of the r.h.s. of \eqref{eq:eqrefeqffxyz110}  we have 
\begin{align}\label{eq:eqexpfxy0}
     &\sum_{k=1}^N \gamma_k\la \nabla f^\tau(z^k), z^k-u\ra 
     =  \sum_{k=1}^N \gamma_k\left( \la \nabla_x f^\tau(x^k,y^k), x^k-x \ra  -  \la \nabla_y  f^\tau(x^k,y^k), y^k-y \ra \right) \notag \\
     &\geq  \sum_{k=1}^N \gamma_k \left( f^\tau(x^k,y^k) - f^\tau(x,y^k)  -   f^\tau(x^k,y^k) + f^\tau(x^k,y)\right) =  \sum_{k=1}^N \gamma_k (f^\tau(x^k,y)  - f^\tau(x,y^k)),
\end{align}
Then we use the fact function $f^\tau(x,y)$ is convex in $x$ and concave in $y$ and  obtain 
\begin{align}\label{eq:fineqE0}
     \left(\sum_{i=1}^{N}\gamma_k\right)^{-1} \sum_{k=1}^N \gamma_k (f^\tau(x^k,y)  - f^\tau(x,y^k))    
     &\geq  f^\tau\left(\hat x^N,y\right)  - f^\tau\left(x, \hat y^N \right) ,
\end{align}
where $ (\hat x^N, \hat y^N)$ is the output of the Algorithm \ref{alg:zoSPA}. 
Using \eqref{eq:fineqE0}  for  \eqref{eq:eqexpfxy0} we get 
\begin{align}\label{eq:fpslff0}
   \sum_{k=1}^N \gamma_k\la \nabla f^\tau(z^k), z^k-u\ra 
     &\geq  \sum_{k=1}^N \gamma_k \left(f^\tau\left( \hat x^N,y\right)  - f^\tau\left(x, \hat y^N \right) \right).
\end{align}
The next lemma is the key moment of the proof giving optimal convergence  result.  
\begin{lemma}\label{lm:dotprodvece}
Let vector $\be$ be a random unit vector from the Euclidean unit sphere $\{\be:\|\be\|_2=1\}$. Then it holds for all $r \in \R^d$
 \begin{equation*}
     \E_{ \be}\left[ \left|\la \be, r \ra \right| \right] \leq {\|r\|_2}/{\sqrt{d}}.
 \end{equation*}
\end{lemma}
Using this  Lemma  \ref{lm:dotprodvece}  we estimate the term $\E_{\be^k} \left[    |\la \be^k, z^k-u\ra|  \right] $ in  \eqref{eq:eqrefeqffxyz110}   
\begin{align}\label{eq:sectermErhs0}
  \E_{\be^k} \left[    \left|\la \be^k, z^k-u\ra\right|  \right] \leq  {\|z^k-u\|_2}/{\sqrt{d}}. 
\end{align}
Now  we substitute  \eqref{eq:fpslff0} and  \eqref{eq:sectermErhs0} to  \eqref{eq:eqrefeqffxyz110}, and get
 under Assumption \ref{ass:err_noise}
 \begin{align}\label{eq:eqrefeqffxyz11220}
   \sum_{k=1}^N \gamma_k   \E_{\be^k,\xi^k}\left[ \la g(z^k,\be^k,\xi^k), z^k-u\ra \right] &\geq
     \sum_{k=1}^{N}\gamma_k  \left(f^\tau\left( \hat x^N,y\right)- f^\tau\left(x, \hat y^N \right)  -  {\sqrt{d}\Delta}{\tau^{-1}}\|z^k-u\|_2\right). 
\end{align}
{\bf Step 3} (under Assumption \ref{ass:err_noise}). Now we combine  \eqref{eq:eqrefeqffxyz11220}   with   \eqref{eq:inmainsqEtwo10} for  \eqref{eq:inmainsqE0} and obtain under Assumption \ref{ass:err_noise} the following
     \begin{align}\label{eq:inmainsqEtwo1330}
     &\sum_{k=1}^{N}\gamma_k \left( f^\tau\left( \hat x^N,y\right)  - f^\tau\left(x, \hat y^N \right) 
     -   {\sqrt{d} \Delta}{\tau^{-1}} \|z^k-u\|_2  \right)\notag \\
     &\leq  V_{z^1}(u) +  \sum_{k=1}^N \frac{\gamma_k^2}{2} \left(c  dM_2^2 a^2_q + {d^2 \Delta^2 }{\tau^{-1}a_q^2}\right).
     \end{align}
Using Lemma \ref{lm:shamirSmoothFun} we obtain
 \begin{align*}
    f^\tau\left( \hat x^N,y\right) -  f^\tau\left(x, \hat y^N \right) 
    &\geq f\left( \hat x^N,y\right) -  f\left(x, \hat y^N \right)  - 2\tau M_2.
 \end{align*}
 Using this we can rewrite \eqref{eq:inmainsqEtwo1330} as follows
 \begin{align}\label{eq:exEeqE10}
       f\left( \hat x^N,y\right) -  f\left(x, \hat y^N \right) 
      &\leq 
     \frac{ V_{z^1}(u)}{\sum_{k=1}^{N}\gamma_k}
      +\frac{c  dM_2^2 a^2_q + d^2 \Delta^2\tau^{-2} a_q^2}{\sum_{k=1}^{N}\gamma_k}\sum_{k=1}^N \frac{\gamma_k^2}{2} \notag \\
      &+{ \sqrt{d} \Delta}{\tau^{-1}} \max_{k}\|z^k-u\|_2 +  2\tau M_2.
 \end{align}
  For the r.h.s. of \eqref{eq:exEeqE10}  we use the definition of  the $\omega$-diameter  of $\Z$:\\
$\D \triangleq \max_{z,v\in \Z}\sqrt{2V_{z}(v)}$ and estimate $\|z^k-u\|_2 \leq \D$ for all $z^1,\dots, Z^k$ and all $u\in \Z$.
 Using this for  \eqref{eq:exEeqE10} and   taking the maximum in $(x,y) \in (\X,\Y)$, we obtain
 \begin{align*}
    \max_{y \in \Y} f\left( \hat x^N,y\right) -  \min_{x\in \X} f\left(x, \hat y^N \right)
     &\leq \frac{ \D^2+ (c  d M_2^2 a^2_q +
     d^2 \Delta^2\tau^{-2}a_q^2)
    \sum_{k=1}^{N} {\gamma_k^2}/{2}}{\sum_{k=1}^{N}\gamma_k}
  \notag \\
  &+{\sqrt{d}\Delta \D}{\tau^{-1}}
     +  2\tau M_2.  
 \end{align*}
Taking the expectation of this and choosing  stepsize $\gamma_k 
=\frac{\D}{M}\sqrt{\frac{2}{N}}$ with $M^2 \triangleq cdM_2^2 a_q^2 + d^2\Delta^2\tau^{-2}a_q^2$  we get
 \begin{align}\label{eq:convergencerate0}
      \E\left[ \max_{y \in \Y} f\left( \hat x^N,y\right) -  \min_{x\in \X} f\left(x, \hat y^N \right)\right]
     &\leq  
     M\D\sqrt{{2}/{N}}
     + {\sqrt{d}\Delta\D}{\tau^{-1}} 
     +  2\tau M_2.  
 \end{align}
\end{proof}


\section{Zeroth-order algorithm with restarts}\label{sec:rest}
In this section, we assume that we  additionally have the $r-$growth condition for duality gap (see, \cite{shapiro2021lectures} for convex optimization problems).  For such a case, we apply the restart technique \cite{juditsky2014deterministic} to  Algorithm \ref{alg:zoSPA} 
\begin{assumption}[$r-$growth condition]\label{ass:strongly_convex}
There is $r \geq 1$ and $\mu_{r} > 0$ such that for all $z = (x,y) \in \Z$
\[\frac{\mu_{r}}{2} \|z-z^{\star}\|_p^{r}
\leq f\left(  x,y^{\star}\right) -  f\left(x^{\star},  y \right),
\]
where $(x^*,y^*)$ is a solution of problem \eqref{eq:min_max_problem}.
\end{assumption}

The next theorem states that if additionally Assumption \ref{ass:strongly_convex} holds, then the convergence results of Theorem \ref{cor:bound_noise} can be improved.
\begin{theorem}\label{th:conv_rate_restarts} 
Let  $\e$ be the desired accuracy to solve problem \eqref{eq:min_max_problem} and  $\tau$  from randomized smoothing \eqref{eq:smoothedfunc} be chosen as $\tau = \O\left(\e/M_2\right)$. Let function $f(x,y,\xi)$  satisfy the Assumption \ref{ass:Lip_obj} and Assumption \ref{ass:strongly_convex} with $r \geq 2$. 
Let   one of the two following statement is true
\begin{enumerate}
    \item Assumption \ref{ass:err_noise} holds with $\Delta \lesssim \frac{\mu_{r}^{1/r}\e^{2-1/r}}{M_2\sqrt{d}}$;
    \item  Assumption \ref{ass:Lip_noise} holds with
$M_{2,\delta} \lesssim \frac{\mu_{r}^{1/r}\e^{1-1/r}}{\sqrt{d}}$.
\end{enumerate}
Then for  $ \hat \e_{\rm sad}  \triangleq   f(\hat x^N, y^{\star}) -    f(x^{\star}, \hat y^N)$, where  $ \hat z^N \triangleq (\hat x^N, \hat y^N)$ is the output of   Algorithm \ref{alg:zoSPA} with restarts, it holds   $\E\left[\hat \e_{\rm sad}(\hat z^N)\right]\leq \e$  after the following number of iterations
 \begin{align}\label{eq:estim_iter_restart}
     N = \widetilde \O\left(\frac{d M_2^2 a_q^2  d}{\mu_{r}^{2/r} \e^{2(r-1)/r}} \right)
 \end{align}
\end{theorem}

\subsection{
Convergence rate in high-probability bound}
Heretofore, all the results were stated in the average, now we provide the convergence results in rems of probability. To do so, we need the following assumption.

\begin{assumption}[Uniformly Lischitz continuity of the objective]\label{ass:uni_Lip_obj}
 Function $f(z,\xi)$ is uniformly $M_2$-Lipschitz continuous in $z\in \Z$ w.r.t. the $\ell_2$-norm, i.e.,   for all  $z_1,z_2\in \Z$ and $\xi \in \Xi$, 
 \[|f(z_1, \xi)- f(z_2,\xi)| \leq M_2 \|z_1- z_2\|_2.\] 
\end{assumption}

The next theorem is stated in the Euclidean proximal setup ($p=q=2, a_2=1$).
\begin{theorem}\label{th:conv_rate_restarts_prob} 
 Let $\e$ be the desired accuracy to solve problem \eqref{eq:min_max_problem} and $\tau$  be chosen as $\tau = \O\left(\e/M_2\right)$.
Let the Assumption \ref{ass:err_noise} holds   with $\Delta \lesssim \frac{\mu_{r}^{1/r}\e^{2-1/r}}{M_2\sqrt{d}}$ and let function $f(x,y,\xi)$  satisfy  Assumption  \ref{ass:uni_Lip_obj}.  Then for the output $ \hat z^N \triangleq (\hat x^N, \hat y^N)$ of   Algorithm \ref{alg:zoSPA},   it holds   $\Prob\left\{\hat{\e}_{\rm sad}(\hat z^N) \leq \e \right\} \geq 1 - \sigma$ after 
\[N =  \O\left( {d M_2^2\D^2}/{\e^{2}} \right)\]
iterations. 
Moreover if Assumption \ref{ass:strongly_convex} is satisfied with $r \geq 1$,  then for the output $ \hat z^N \triangleq (\hat x^N, \hat y^N)$ of   Algorithm \ref{alg:zoSPA} with restarts,   it holds    $\Prob\left\{\hat{\e}_{\rm sad}(\hat z^N) \leq \e \right\} \geq 1 - \sigma$ after the following number of iterations
 \begin{align}\label{eq:estim_iter_restart_prob}
     N = \widetilde \O\left(\frac{dM_2^2 }{\mu_{r}^{2/r} \e^{2(r-1)/r}} \right).
 \end{align}
\end{theorem}

\begin{remark}\label{hpb}
In the case $r=2$ it is probably possible to improve the bound \eqref{eq:estim_iter_restart_prob} in $\log \e^{-1}$ factor by using alternative algorithm \cite{harvey2019tight}. 
\end{remark}

\section{Infinite noise variance}\label{sec:inf_noise}
Now we comment on the case when 
 the second moment of the  stochastic subgradient  $\nabla f(z,\xi)$ is unbounded. In this case the rate of convergence may changes dramatically. For such a case, we modify the Assumptions~\ref{ass:Lip_obj}
 \begin{assumption}[Lipschitz continuity of the objective under infinite noise variance]\label{ass:Lip_objInf}
 Function $f(z,\xi)$ is $M_2$-Lipschitz continuous in $z\in \Z$ w.r.t. the $\ell_2$-norm, i.e.,   for all  $z_1,z_2\in \Z$ and $\xi \in \Xi$, 
 \[|f(z_1, \xi)- f(z_2,\xi)| \leq M_2(\xi) \|z_1- z_2\|_2.\] 
 Moreover, there exists a positive constant $\tilde{M_2}$ such that  $  \E \left[M_2(\xi)^{1+\kappa}\right]\le \tilde{M}_2^{1+\kappa}$, where $\kappa \in \left(0,1\right]$.
\end{assumption}
If Assumption \ref{ass:Lip_objInf} holds, the stepsize in Algorithm \ref{alg:zoSPA} is replaced by 
\[\gamma_k = \frac{\left((1+\kappa)V_{z^1}(z^{\star})/\kappa\right)^{\frac{1}{1+\kappa}}}{\tilde{M}}N^{-\frac{1}{1+\kappa}},\]
where 
$V_{z^1}(z^{\star})$ is the Bregman divergence determined by the following prox-function with 
 $q \in \left[1+\kappa,\infty\right)$ and $1/p + 1/q = 1$
\[\omega(x) = K_q^{1/\kappa} \frac{\kappa}{1+\kappa}\|x\|_p^{\frac{1+\kappa}{\kappa}} \mbox{ where } K_q = 10\max\left\{1,(q-1)^{(1+\kappa)/2}\right\}.  \]
and constant $\tilde{M}$ is determined as (can be obtained from \citet[Lemmas 9 -- 11]{shamir2017optimal}) 
 \begin{enumerate}[leftmargin=*]
     \item  under Assumption  \ref{ass:err_noise} $
        \E\left[\|g(z,\xi,\be)\|^{1+\kappa}_q\right] \leq \tilde{c} a^2_q d^{(1+\kappa)/2}\tilde{M}_2^{1+\kappa} + 2^{1+\kappa}{d^{1+\kappa} a_q^2\Delta^2}{\tau^{-2}}=\tilde{M}^{1+\kappa}, $
     \item  under Assumption \ref{ass:Lip_noise} 
 $
    \E\left[\|g(z,\xi,\be)\|^{1+\kappa}_q\right] \leq \tilde{c} a^2_q d^{(1+\kappa)/2} (\tilde{M}_2^{1+\kappa}+M_{2,\delta}^{1+\kappa}) = \tilde{M}^{1+\kappa},$
 \end{enumerate}
 where    $\tilde{c}$ is some numerical constant and
 $\sqrt{\E_{\be} \left[ \|\be\|_q^{2+2\kappa}\right]} \leq \tilde{a}_q^2$. As a particular case: $\tilde{a}_{2}^2 = 1$, $\tilde{a}_{\infty}^2 = \O\left({\frac{(\log d)^{(1+\kappa)/2}}{d^{(1+\kappa)/2}}}\right)$.

Based on \cite{vural2022mirror} one can prove 
that convergence rate of Algorithm \ref{alg:zoSPA} changes dramatically
in comparison with  Theorem~\ref{th:conv_rate_restarts} (see \eqref{eq:convergencerate0}), namely under Assumption  \ref{ass:err_noise}
\[
      \E\left[ \max_{y \in \Y} f\left( \hat x^N,y\right) -  \min_{x\in \X} f\left(x, \hat y^N \right)\right]
     \leq  
    \tilde{M}\left(\frac{1+\kappa}{\kappa}V_{z^1}(z^{\star})\right)^{\frac{\kappa}{1+\kappa}} N^{-\frac{\kappa}{1+\kappa}}
     + \frac{\Delta\D\sqrt{d}}{\tau} 
     +  2\tau \tilde{M_2}.  
\]
These results can be further generalized to $r-$growth condition for duality gap ($r\ge 2$).

\section{`Upper' bound in the case of Lipschitz noise}\label{sec:upper_bound}
Now let us consider a stochastic convex optimization problem of the form
\begin{equation}\label{eq:conv_prob}
    \min_{x \in \X} F(x) \triangleq \E_{\xi} f(x,\xi),
\end{equation}
where  $X \subseteq \R^{d}$ is a convex set, and for all $\xi$, $f(x,\xi)$ is convex in $x\in\X$  and satisfies Assumption \ref{ass:Lip_obj}. The empirical counterpart of this problem \eqref{eq:conv_prob} is
\begin{equation}\label{emp_prob}
    \min_{x \in \X} \hat{F}(x)\triangleq \frac{1}{N}\displaystyle\sum_{k = 1}^{N} f(x, \xi^k).
\end{equation}
The exact solution ($\nicefrac{\e}{2}$-solution)  of \eqref{emp_prob} is an $\e$-solution of \eqref{eq:conv_prob} if the sample size $N$ is taken as follows  \cite{shapiro2005complexity,feldman2016generalization}
\begin{equation}\label{LB}
   N=\widetilde{ \Omega}\left(d {M_2^2 \D_2^2}/{\e^{2}}\right),
\end{equation}
where $\D_2$ is the diameter of $X$ in the $\ell_2$-norm. This lower bound is tight \cite{shapiro2005complexity,shapiro2021lectures}. On the other hand,  $\hat{F}(x)$ from \eqref{emp_prob} can be considered as an inexact zeroth-order oracle for $F(x)$ from \eqref{eq:conv_prob}. If $\delta(x) = F(x) - \hat{F}(x)$ then  in $\text{Poly}\left(d,1/\e\right)$ points $y,x$ with probability $1 - \beta$ the following holds 
$$\left|\delta(y) - \delta(x)\right| = \O\left(\frac{M_2\|y-x\|_2}{\sqrt{N}}\ln\left(\frac{\text{Poly}\left(d,1/\e\right)}{\beta}\right)\right)= \widetilde{\O}\left(\frac{M_2\|y-x\|_2}{\sqrt{N}}\right),$$
i.e., $\delta(x)$ is a Lipschitz function with Lipschitz constant 
\begin{equation}\label{delta}
    {M_{2,\delta}} = \widetilde{\O}\left({M_2}/{\sqrt{N}}\right).
\end{equation}

Let us assume there exists a zeroth-order algorithm that can solve \eqref{eq:conv_prob} with accuracy  $\e$ in $\text{Poly}\left(d,1/\e\right)$ oracle calls, where an oracle returns an inexact value of $F(x)$ with a noise that has the following Lipschitz constant 
$$M_{2,\delta}  \gg \frac{\e}{\D_2\sqrt{d}}.$$ 
We can use this algorithm to solve problem \eqref{eq:conv_prob} with $N$ determines from (see \eqref{delta})
$$\frac{M_2}{\sqrt{N}} \gg \frac{\e}{\D_2\sqrt{d}}, \qquad \mbox{i.e.} \qquad N \ll d{M_2^2 \D_2^2}/{\e^2},$$
that contradicts  the lower bound \eqref{LB}. Thereby it is impossible in general to solve with accuracy $\e$ (in function value) Lipschitz convex  optimization problem via $\text{Poly}\left(d,1/\e\right)$ inexact zero-order oracle calls if Lipschitz constant of noise is greater than  
\begin{equation}\label{LBN}
   {\e}/({\D_2\sqrt{d}}).
\end{equation}

Unfortunately, we obtain this upper bound assuming that $\text{Poly}\left(d,1/\e\right)$ points were chosen regardless of $\left\{\xi^k\right\}_{k=1}^N$. That is not the case for the most of practical algorithms, in particular, considered above. But the dependence of these points from $\left\{\xi^k\right\}_{k=1}^N$ is significantly weakened by randomization we use in zero-order methods. So we may expect that nevertheless this upper bound still takes place. 

For arbitrary $\text{Poly}\left(d,1/\e\right)$ points (possibly that could depend on $\left\{\xi^k\right\}_{k=1}^N$) we can guarantee only (see \cite{shapiro2021lectures})
\begin{equation}\label{delta_}
    M_{2,\delta}  = \widetilde{\O}\left(\sqrt{d}{M_2}/{\sqrt{N}}\right).
\end{equation}
That is large than \eqref{delta}. Consequently, the upper bound \eqref{LBN} should be 
rewritten as
\begin{equation}\label{LBN_}
    {\e}/{\D_2}.
\end{equation}

We believe that \eqref{LBN_} is not tight upper bound, rather than \eqref{LBN}. That is, there exists algorithm (see Algorithm \ref{alg:zoSPA}) that can solve \eqref{eq:conv_prob} with accuracy $\e$ (in function value) via $\simeq dM_2^2\D_2^2/\e^2$ inexact zero-order oracle calls if Lipschitz constant of noise is bounded from above by $M_{2,\delta}  \lesssim \e/(\D_2\sqrt{d})$. But there are no algorithms reaching the same accuracy $\e$ by using $\text{Poly}\left(d,1/\e\right)$ inexact zero-order oracle calls if Lipschitz constant of noise is bounded from above by $M_{2,\delta}  \gg \e/(\D_2\sqrt{d})$, in particular, for $M_{2,\delta} $ given by \eqref{LBN_}. Note, that \eqref{LB} holds  also if $f(x,\xi)$ has Lipschitz $x$-gradient \cite{feldman2016generalization}. Hence we can expect that obtained  lower bounds take place also for smooth problems.






\section*{Conclusion}
In this paper, we demonstrate how to solve  non-smooth stochastic convex-concave saddle point problems  with two-point gradient-free oracle. In the Euclidean proximal setup, we obtain  optimal oracle complexity bound   $\O(d/\e^2)$.    We also generalize this result for an arbitrary proximal setup and obtain a tight upper bound  $\O(\e^2/\sqrt{d})$ on maximal level of additive adversary noise in oracle calls. We generalize this result for the class of saddle point problems satisfying $r$-growth condition for duality gap.

\section*{Acknowledgments}
This work was supported by a grant for research centers in the field of 
artificial intelligence, provided by the Analytical Center for the 
Government of the Russian Federation in accordance with the subsidy 
agreement (agreement identifier 000000D730321P5Q0002 ) and the agreement 
with the Ivannikov Institute for System Programming of the Russian 
Academy of Sciences dated November 2, 2021 No. 70-2021-00142.

{\small

\bibliographystyle{apalike}
\bibliography{ref}
}

\newpage
\appendix


\section{The Basic Idea and Possible Generalization}
In this section, we give some possible generalizations
for our results. For simplicity, we consider a non-stochastic non-smooth convex optimization problem in the Euclidean proximal setup
\begin{equation}\label{prblm}
    \min_{x\in \X\subseteq \mathbb{R}^d} f(x),
\end{equation}
where $X \subseteq \R^{d}$ is a convex set and $f$ is  $M$-Lipschitz continuous.
In this problem, we replace the non-smooth objective $f(x)$ by its smooth  approximation $f^{\tau}(x)$ defined in  \eqref{eq:smoothedfunc}, i.e.:
$
f^{\tau}(x)\triangleq  \mathbb{E}_{\tilde{e}} f(x + \tau \tilde{e})
$, where  $\tau > 0$ is some constant and $\tilde{e}$ is a  vector 
picked uniformly at random from the Euclidean unit ball $\{\tilde{e}:\|\tilde{e}\|_2\leq 1\}$.
From \cite{duchi2015optimal} it follows that
\begin{equation}\label{apprx}
f(x) \leq f^{\tau}(x) \leq f(x) + \tau M.
\end{equation}
Let the objective $f(x)$    can be not directly observed but instead, its noisy approximation $\vp(x) \triangleq f(x) + \delta(x)$ can be queried, where $\delta(x)$ is some adversarial noise such that $\|\delta(x)\| \leq \Delta$.
Similarly to  \eqref{eq:grad_approx} we can estimate the gradient of $\vp(x)$ d by the following gradient-free approximation 
\begin{equation*}\label{sg}
  g(x,e)    = \frac{d}{2\tau}\left(\vp(x+\tau e) - \vp(x-\tau e)\right)e,
\end{equation*}
 Due to 
\cite{gasnikov2017stochastic} (see also Lemma~\ref{lm:dotprodvece}) for all $r \in \R^d$
\begin{equation}\label{unbias}
 \mathbb{E}_e  \left[ \langle g(x, e)  - \nabla f^{\tau}(x), r\rangle \right]\lesssim \sqrt{d}{\Delta\|r\|_2}{\tau^{-1}}
 \end{equation}
 and \cite{shamir2017optimal,beznosikov2020gradient}
 \begin{equation}\label{var}
   \mathbb{E}_e \left[ \| g(x, e) - \mathbb{E}_e \nabla f^{\tau}(x, e) \|^2_2 \right]\simeq \mathbb{E}_e \left[ \| g(x, e) \|^2_2 \right] \lesssim 
   d M^2 + {d^2\Delta^2}{\tau^{-2}} ,
    \end{equation}
where $e$ is a random vector uniformly distributed on    the Euclidean unit sphere $\{e:\|e\|_2= 1\}$.  Bound  \eqref{unbias} is $\sqrt{d}$ better than in the bounds from \cite{beznosikov2020gradient,akhavan2020exploiting}. Moreover, for $M_{2,\delta}$-Lipschitz noise, Eq. \eqref{unbias} can be rewritten as follows
\begin{equation}\label{unbias_2}
 \mathbb{E}_e \left[ \langle  g(x, e)  - \nabla f^{\tau}(x), r\rangle \right] \lesssim \sqrt{d}M_{2,\delta}\|r\|_2.
 \end{equation}
We will say that an algorithm $\bf A$ (with $g(x, e)$ oracle) is \textit{robust} for  $f^{\tau}$ if the bias in the l.h.s. of \eqref{unbias} does not accumulate over method iterations. That is, if for $\bf A$ with $\Delta = 0$ 
$$\mathbb{E}f^{\tau}(x^N) - \min_{x\in \X} f^{\tau}(x)\le \Theta_A(N),$$
then with $\Delta > 0$ and (variance control)  ${d^2\Delta^2}{\tau^{-2}} \lesssim d M^2$, see \eqref{var}:  
\begin{equation}\label{ns}
    \mathbb{E} f^{\tau}(x^N) - \min_{x\in \X} f^{\tau}(x)= O\left(\Theta_A(N) + \sqrt{d}{\Delta\D}{\tau^{-1}}\right),
\end{equation}
 where $\D$ is a diameter of $\X$. Here  $N$ should be taken such that the first term of the r.h.s. of Eq. \eqref{ns}  is not smaller than the second one. Similar definition can be made for \eqref{unbias_2}. 
 Many known methods are robust, e.g., stochastic versions of mirror descent, mirror prox, gradient method and fast gradient  method are robust
\cite{juditsky2012first-order,d2008smooth,cohen2018acceleration,dvinskikh2020accelerated,dvinskikh2021decentralized,gorbunov2019optimal}. \\
Now we explain how to obtain the bound on the level of noise $\Delta$.

\textbf{Approximation.}
To approximate non-smooth function $f(x)$ by smooth function $f^{\tau}(x)$, constant $\tau$ should be taken as follows: $\tau = \frac{\e}{2M}$ (see \eqref{apprx}).

\textbf{Variance control.} To control the variance and obtain 
the second moment of the stochastic gradient with $\Delta>0$ and $\Delta=0$ of the same order, $\Delta$ should be taken not bigger that (see \eqref{var}):
$\Delta\lesssim\frac{\tau M}{\sqrt{d}}.$

\textbf{Bias.} From \eqref{ns} we will have more restrictive condition on the level of noise:
$\Delta\lesssim\frac{\tau\e}{\D\sqrt{d}}.$
Combining  the bias condition and approximation condition leads to the following bound on the $\Delta$:
\begin{equation}\label{lon}
\Delta \lesssim\frac{\e^2}{\D M\sqrt{d}}.
\end{equation}
For Lipschitz noise and for saddle point problems, the same bound holds by the same reasoning.

More interestingly, as stochastic (batched) versions of fast gradient method and mirror prox are also robust \cite{gorbunov2019optimal,stonyakin2021inexact} and $f^{\tau}$ has $\nicefrac{\sqrt{d}M}{\tau}$-Lipschitz gradient \cite{duchi2012randomized,gasnikov2022power}, we can improve the results of this  paper by using parallelized smoothing technique from \cite{gasnikov2022power}. For instance, for non-smooth convex optimization problem \eqref{prblm}, the gradient-free method from \cite{gasnikov2022power} (with optimal number of oracle calls and the best known number of consistent iterations) provides also the highest possible level of noise, given by \eqref{lon}.

\section{Proofs of auxiliary lemmas}\label{sec:aux_lemmas}

\textit{Proof of Lemma \ref{lm:shamirEnabla}. }
 Let us consider
\begin{align}\label{eq:sqnormE}
    &\E_{\xi,\be} \left[\|g(z,\xi,\be)\|^2_q \right] = \E_{\xi,\be}\left[ \left\|\frac{d}{2\tau}\left(\vp(z+\tau \be,\xi) - \vp(z-\tau \be,\xi)   \right) \be\right\|^2_q \right] \notag \\
    &= \frac{d^2}{4\tau^2}\E_{\xi,\be} \left[ \|\be\|_q^2 \left(f(z+\tau \be,\xi) +\delta(z+\tau\be) - f(z-\tau \be,\xi) -\delta(z-\tau\be)  \right)^2 \right] \notag\\
     &\leq \frac{d^2}{2\tau^2}\left(\E_{\xi,\be} \left[\|\be\|_q^2 \left(f(z+\tau \be,\xi) - f(z-\tau \be,\xi)\right)^2\right] + \E_{\be} \left[ \|\be\|_q^2 \left(\delta(z+\tau\be)  -\delta(z-\tau\be)  \right)^2 \right] \right),
\end{align}
where we used that for all $ a,b, (a-b)^2\leq 2a^2+2b^2$.
For the first term in the r.h.s. of \eqref{eq:sqnormE}, the following holds with an arbitrary parameter $\alpha$
\begin{align}\label{eq:fgkbbhgdddgh}
 & \frac{d^2}{2\tau^2}\E_{\xi,\be} \left[\|\be\|_q^2 \left(f(z+\tau \be,\xi) - f(z-\tau \be,\xi)\right)^2\right]  \notag\\
  &= \frac{d^2}{2\tau^2}\E_{\xi,\be} \left[ \|\be\|_q^2 \left((f(z+\tau \be,\xi) -\alpha) -(f(z-\tau \be,\xi) - \alpha)   \right)^2 \right] \notag\\
    &\leq \frac{d^2}{\tau^2}\E_{\xi,\be} \left[ \|\be\|_q^2 (f(z+\tau \be,\xi)-\alpha)^2 + (f(z-\tau \be,\xi) -\alpha)^2 \right] \qquad /* \forall a,b, (a-b)^2\leq 2a^2+2b^2 */ \notag\\
     &=  \frac{d^2}{\tau^2}\left(\E_{\xi,\be} \left[ \|\be\|_q^2 (f(z+\tau \be,\xi)-\alpha)^2 \right] + \E_{\xi,\be} \left[\|\be\|_q^2(f(z-\tau \be,\xi) -\alpha)^2 \right] \right) \notag\\
    &= \frac{2d^2}{\tau^2}\E_{\xi,\be} \left[ \|\be\|_q^2 (f(z+\tau \be,\xi)-\alpha)^2 \right]. \qquad /* \mbox{ the distribution of $\be$ is  symmetric  } */
\end{align}
Applying the Cauchy--Schwartz inequality for  \eqref{eq:fgkbbhgdddgh} and using $\sqrt{\E \left[ \|\be\|_q^4\right]} \leq a_q^2$ we obtain 
\begin{align}\label{eq:funcsq}
     \frac{d^2}{\tau^2}\E_{\xi,\be} \left[ \|\be\|_q^2 (f(z+\tau \be,\xi)-\alpha)^2 \right] 
     &\leq \frac{d^2}{\tau^2} \E_{\xi} \left[\sqrt{\E \left[ \|\be\|_q^4\right]} \sqrt{\E_{\be} \left[ f(z+\tau \be,\xi)-\alpha)^4\right]} \right]\notag\\
     &\leq \frac{d^2 a_q^2}{\tau^2} \E_{\xi}\left[\sqrt{\E_{\be} \left[ f(z+\tau \be,\xi)-\alpha)^4\right]}\right] . 
\end{align}
Next we use the following lemma.
\begin{lemma}\citet[Lemma 9]{shamir2017optimal}\label{lm:shamirLm9}
For any function $f(\be)$ which is $M$-Lipschitz w.r.t. the $\ell_2$-norm, it holds that if $\be$ is uniformly distributed on the Euclidean unit sphere, then for some  constant $c$
\[
\sqrt{\E\left[(f(\be) - \E f(\be))^4 \right]} \leq cM_2^2/d.
\]
\end{lemma}
Then we  use this Lemma \ref{lm:shamirLm9} along with the fact that $f(z+\tau\be,\xi)$ is $\tau M_2(\xi)$-Lipschitz, w.r.t. $\be$ in terms of the $\ell_2$-norm. Thus for  \eqref{eq:funcsq} and $\alpha \triangleq \E_{\be}\left[f(z+\tau \be,\xi)\right]$, it holds  
\begin{align}\label{eq:finlhkjfsvn}
    \frac{d^2 a_q^2}{\tau^2} \E_{\xi}\left[\sqrt{\E_{\be} \left[ f(z+\tau \be,\xi)-\alpha)^4\right]}\right]
    &\leq \frac{d^2 a_q^2}{\tau^2} \cdot \frac{c \tau^2 \E\left[ M_2^2(\xi)\right]}{d} = c d  M_2^2 a_q^2.  
\end{align}
\begin{enumerate}[leftmargin=*]
    \item Under the Assumption \ref{ass:err_noise}
     for the second term in the r.h.s. of   \eqref{eq:sqnormE}, the following holds
    \begin{align}\label{eq:fgkssssffdgh}
        &\frac{d^2}{4\tau^2}\E_{\be} \left[ \|\be\|_q^2 \left(\delta(z+\tau\be)  -\delta(z-\tau\be)  \right)^2 \right] \leq  \frac{d^2\Delta^2}{\tau^2} \E \left[ \|\be\|_q^2  \right]\leq \frac{d^2 \Delta^2 a_q^2}{\tau^2}.
    \end{align}
    \item Under the Assumption \ref{ass:Lip_noise}
    for the second term in the r.h.s. of   \eqref{eq:sqnormE}, the following holds with an arbitrary parameter $\beta$
    \begin{align}\label{eq:noisesq}
        &\frac{d^2}{4\tau^2}\E_{\be} \left[ \|\be\|_q^2 \left(\delta(z+\tau\be)  -\delta(z-\tau\be)  \right)^2 \right] \notag \\
     &= \frac{d^2}{2\tau^2}\E_{\be} \left[ \|\be\|_q^2 \left((\delta(z+\tau\be) -\beta) -(\delta(z-\tau\be)  -\beta) \right)^2 \right] \notag\\
    &\leq \frac{d^2}{\tau^2}\E_{\be} \left[ \|\be\|_q^2 \left((\delta(z+\tau\be)-\beta)^2  +(\delta(z-\tau\be) -\beta)^2 \right) \right] \quad /* \forall a,b, (a-b)^2\leq 2a^2+2b^2 */ \notag\\
    &= \frac{d^2}{2\tau^2}\left(\E_{\be} \left[ \|\be\|_q^2 (\delta(z+\tau \be)-\beta)^2 \right] + \E_{\be} \left[\|\be\|_q^2(\delta(z-\tau \be) -\beta)^2 \right] \right) \notag\\
    &= \frac{d^2}{\tau^2}\E_{\be} \left[ \|\be\|_q^2 \left(\delta(z+\tau\be)-\beta  \right)^2 \right]. \qquad /* \mbox{ the distribution of $\be$ is  symmetric  } */ \notag \\
     &\leq \frac{d^2}{\tau^2}\sqrt{\E \left[ \|\be\|_q^4\right]} \sqrt{\E_{\be} \left[ \left(\delta(z+\tau\be)-\beta  \right)^4 \right] } \leq \frac{d^2 a_q^2}{\tau^2} \sqrt{\E_{\be} \left[ \left(\delta(z+\tau\be)-\beta  \right)^4 \right]}.
    \end{align}
\end{enumerate}
Then we  use Lemma \ref{lm:shamirLm9} together with the fact that $\delta(z+\tau\be)$ is $\tau M_{2,\delta}$-Lipschitz continuous, w.r.t. $\be$ in terms of the $\ell_2$-norm. Thus, for  \eqref{eq:noisesq}  and  $\beta \triangleq \E_{\be}\left[\delta(z+\tau \be)\right]$, the following holds
\begin{align}\label{eq:fkssoomnbd}
    \frac{d^2 a_q^2}{\tau^2}  \sqrt{\E_{\be} \left[ \left(\delta(z+\tau\be)-\beta  \right)^4 \right] } 
    &\leq \frac{d^2 a_q^2}{\tau^2}  \cdot \frac{c \tau^2  M_{2,\delta}^2}{d} \leq c a^2_q d M_{2,\delta}^2. 
\end{align}
Using  \eqref{eq:finlhkjfsvn} and \eqref{eq:fgkssssffdgh} (or  \eqref{eq:fkssoomnbd}) for  \eqref{eq:sqnormE}, we get the statement of the lemma.\\
\qed

\textit{Proof of Lemma \ref{lm:expecapproxgrad}. }
 Let us consider
\begin{align*}
  \hspace{-0.3cm} g(z,\be,\xi) 
    &= \frac{d}{2\tau}   \left(\vp(z+\tau \be,\xi) - \vp(z-\tau \be,\xi)   \right) \be \notag \\
   &=\frac{d}{2\tau} \left(\left(f(z+\tau \be,\xi) - f(z-\tau \be,\xi) \right) \be +     (\delta(z+\tau \be) - \delta(z-\tau \be)) \be\right).
\end{align*}
Using this we have the following \begin{align}\label{eq:sqprod}
    \E_{\xi, \be}\left[\la g(z,\xi,\be),r\ra\right]   &=\frac{d}{2\tau}  \E_{\xi, \be}\left[\la\left(f(z+\tau \be,\xi) - f(z-\tau \be,\xi) \right) \be, r \ra \right]\notag \\
    &+\frac{d}{2\tau}  \E_{ \be}\left[     \la (\delta(z+\tau \be) - \delta(z-\tau \be)) \be, r\ra\right].
\end{align}
Taking the expectation, w.r.t. $\be$, from  the first term of the r.h.s. of \eqref{eq:sqprod} and  using the symmetry of the distribution of $\be$, we have
 \begin{align}\label{eq:nablaeqfunc}
      &\frac{d}{2\tau}\E_{\xi, \be} \left[  \left(\la f(z+\tau \be,\xi)  - f(z-\tau \be,\xi) \right) \be, r\ra \right] 
      \notag \\
      &= \frac{d}{2\tau}\E_{\xi, \be}  \left[ \la f(z+\tau \be,\xi)\be, r\ra \right] +\frac{d}{2\tau}\E_{\xi, \be}  \left[  \la f(z-\tau \be,\xi)\be, r\ra \right]   \notag\\
         &= \frac{d}{\tau}\E_{\be} \left[ \la\E_{\xi} \left[ f(z+\tau \be,\xi)\right] \be, r\ra\right] = \frac{d}{\tau}\E_{\be} \left[ \la  f(z+\tau \be) \be, r\ra\right] =\la \nabla f^\tau(z),r\ra. \quad /*  \mbox{Lemma \ref{lm:flax}}  */ 
\end{align}
\begin{enumerate}[leftmargin=*]
    \item For the second term of the r.h.s. of \eqref{eq:sqprod} under the Assumption \ref{ass:err_noise} we obtain 
    \begin{align}\label{eq:secondtermlipschnass2}
    \frac{d}{2\tau}  \E_{ \be} \left[  \la(\delta(z+\tau \be) - \delta(z-\tau  \be)) \be, r\ra \right] \geq - \frac{d}{2\tau} 2\Delta \E_{\be} \left[ \left|\la \be, r \ra \right|\right] = -\frac{d\Delta}{\tau} \E_{\be} \left[ \left|\la \be, r \ra \right| \right].
\end{align}
    \item For the second term of the r.h.s. of \eqref{eq:sqprod} under the Assumption \ref{ass:Lip_noise} we obtain  
\begin{align}\label{eq:secondtermlipschn}
    \frac{d}{2\tau}  \E_{\be} \left[ \la (\delta(z+\tau \be) - \delta(z-\tau \be)) \be, r\ra\right] \geq - \frac{d}{2\tau} M_{2,\delta} 2\tau \E_{\be} \left[ \|\be\|_2 \left|\la\be, r\ra \right|\right] = -d M_{2,\delta}  \E_{\be} \left[ \left|\la\be, r\ra\right|\right].
\end{align}
\end{enumerate}
Using  \eqref{eq:nablaeqfunc} and \eqref{eq:secondtermlipschnass2} (or  \eqref{eq:secondtermlipschn}) for  \eqref{eq:sqprod} we get the statement of the lemma.\\
\qed

\section{Complete proof of Theorem \ref{cor:bound_noise}.}\label{sec:compl_proof_main_th}

\textit{Proof of Theorem \ref{cor:bound_noise}. }
By the definition $z^{k+1}= {\rm Prox}_{z^k}\left(\gamma_k g(z^k,\be^k,\xi^k) \right)$ we get  \cite{ben2001lectures}, for all $u\in \Z$
\begin{align*}
   \gamma_k\la g(z^k,\be^k,\xi^k), z^k-u\ra \leq V_{z^k}(u) - V_{z^{k+1}}(u) + \gamma_k^2\|g(z^k,\be^k,\xi^k)\|_q^2/2.   
\end{align*}
Taking the conditional  expectation w.r.t. $\xi, \be$ and summing for $k=1,\dots,N$  we obtain, for all $u\in \Z$
\begin{align}\label{eq:inmainsqE}
 \sum_{k=1}^N \gamma_k  \E_{\be^k,\xi^k} \left[ \la g(z^k,\be^k,\xi^k), z^k-u\ra\right] 
   &\leq  V_{z^1}(u)  + \sum_{k=1}^N\frac{\gamma_k^2}{2} \E_{\be^k,\xi^k}\left[\|g(z^k,\be^k,\xi^k)\|_q^2 \right]. 
\end{align}
{\bf Step 1.}\\
For the second term in the r.h.s of inequality \eqref{eq:inmainsqE} we use Lemma \ref{lm:shamirEnabla} and obtain
 \begin{enumerate}[leftmargin=*]
     \item  under Assumption  \ref{ass:err_noise}:
     \begin{equation}\label{eq:inmainsqEtwo1}
        \E_{\be^k,\xi^k}\left[\|g(z^k,\xi^k,\be^k)\|^2_q\right] \leq c dM_2^2 a^2_q  + {d^2 \Delta^2}{\tau^{-2}a_q^2}, 
     \end{equation}
     \item  under Assumption \ref{ass:Lip_noise}: 
\begin{align}\label{eq:inmainsqEtwo}
    \E_{\be^k,\xi^k}\left[\|g(z^k,\xi^k,\be^k)\|^2_q\right] \leq c  d (M_2^2+M_{2,\delta}^2)a^2_q,
\end{align}
 \end{enumerate}
 where    $c$ is some numerical constant and $\sqrt{\E \left[ \|\be^k\|_q^4\right]} \leq a_q^2$.

{\bf Step 2.}\\
For the l.h.s. of    \eqref{eq:inmainsqE}, we use Lemma \ref{lm:expecapproxgrad} with  $u \triangleq (x^\top,y^\top)^\top$

 \begin{enumerate}[leftmargin=*]
 \item under Assumption \ref{ass:err_noise}
 \begin{align}\label{eq:eqrefeqffxyz11}
     \sum_{k=1}^N \gamma_k \E_{\be^k,\xi^k}\left[\la g(z^k,\be^k,\xi^k), z^k-u\ra \right]
     &\geq  \sum_{k=1}^N \gamma_k \la \nabla f^\tau(z^k), z^k-u \ra  \notag \\ &+\sum_{k=1}^N \gamma_k  \E_{\be^k}\left[ \left|\la  {d\Delta}{\tau^{-1}}  \be^k, z^k-u  \ra \right|\right].
\end{align}
     \item  under Assumption \ref{ass:Lip_noise}      
\begin{align}\label{eq:eqrefeqffxyz}
     \sum_{k=1}^N \gamma_k \E_{\be^k,\xi^k}\left[\la g(z^k,\be^k,\xi^k), z^k-u\ra \right] 
     &\geq  \sum_{k=1}^N \gamma_k\la \nabla f^\tau(z^k), z^k-u\ra \notag \\ 
     &+\sum_{k=1}^N \gamma_k \E_{\be^k}\left[ \left|\la  dM_{2,\delta} \be^k, z^k-u\ra\right| \right].
\end{align}
 \end{enumerate}
For the first term of the r.h.s. of \eqref{eq:eqrefeqffxyz11} and \eqref{eq:eqrefeqffxyz} we have 
\begin{align}\label{eq:eqexpfxy}
     \sum_{k=1}^N \gamma_k\la \nabla f^\tau(z^k), z^k-u\ra 
     &= \sum_{k=1}^N \gamma_k \left\la \begin{pmatrix} \nabla_x f^\tau(x^k, y^k) \\ -\nabla_y  f^\tau(x^k,y^k) \end{pmatrix}, \begin{pmatrix} x^k-x \\ y^k-y \end{pmatrix} \right\ra \notag \\
     &=  \sum_{k=1}^N \gamma_k\left( \la \nabla_x f^\tau(x^k,y^k), x^k-x \ra  -  \la \nabla_y  f^\tau(x^k,y^k), y^k-y \ra \right) \notag \\
     &\geq  \sum_{k=1}^N \gamma_k ( f^\tau(x^k,y^k) - f^\tau(x,y^k))  -  ( f^\tau(x^k,y^k) - f^\tau(x^k,y)) \notag \\
     &=  \sum_{k=1}^N \gamma_k (f^\tau(x^k,y)  - f^\tau(x,y^k)).
\end{align}
Then we use the fact function $f^\tau(x,y)$ is convex in $x$ and concave in $y$ and  obtain 
\begin{align}\label{eq:fineqE}
     \frac{1}{\sum_{i=1}^{N}\gamma_k} \sum_{k=1}^N \gamma_k (f^\tau(x^k,y)  - f^\tau(x,y^k))  &\geq  f^\tau\left(\frac{\sum_{k=1}^N \gamma_k x^k}{\sum_{k=1}^{N}\gamma_k},y\right)  - f^\tau\left(x, \frac{\sum_{k=1}^N \gamma_k y^k}{\sum_{k=1}^{N}\gamma_k}\right)\notag \\
     &=  f^\tau\left(\hat x^N,y\right)  - f^\tau\left(x, \hat y^N \right) ,
\end{align}
where $ (\hat x^N, \hat y^N)$ is the output of the Algorithm \ref{alg:zoSPA}. 
Using \eqref{eq:fineqE}  for  \eqref{eq:eqexpfxy} we get 
\begin{align}\label{eq:fpslff}
     \sum_{k=1}^N \gamma_k\la \nabla f^\tau(z^k), z^k-u\ra 
     &\geq \sum_{k=1}^{N}\gamma_k \left( f^\tau\left( \hat x^N,y\right)  - f^\tau\left(x, \hat y^N \right) \right).
\end{align}
Next we estimate the term $\E_{\be^k} \left[    |\la \be^k, z^k-u\ra|  \right] $ in  \eqref{eq:eqrefeqffxyz11} and \eqref{eq:eqrefeqffxyz},   by the Lemma  \ref{lm:dotprodvece} 
\begin{align}\label{eq:sectermErhs}
  \E_{\be^k} \left[    \left|\la \be^k, z^k-u\ra\right|  \right] \leq  {\|z^k-u\|_2}/{\sqrt{d}}. 
\end{align}
Now  we substitute   \eqref{eq:fpslff} and  \eqref{eq:sectermErhs} to  \eqref{eq:eqrefeqffxyz11} and \eqref{eq:eqrefeqffxyz}, and get
 \begin{enumerate}[leftmargin=*]
 \item under Assumption \ref{ass:err_noise}
 \begin{align}\label{eq:eqrefeqffxyz1122}
   \sum_{k=1}^N \gamma_k   \E_{\be^k,\xi^k}\left[ \la g(z^k,\be^k,\xi^k), z^k-u\ra \right] &\geq
     \sum_{k=1}^{N}\gamma_k  \left(f^\tau\left( \hat x^N,y\right)- f^\tau\left(x, \hat y^N \right) -   {\sqrt{d}\Delta\|z^k-u\|_2}{\tau^{-1}} \right). 
\end{align}
     \item  under Assumption \ref{ass:Lip_noise}      
\begin{align}\label{eq:eqrefeqffxyz22}
     \sum_{k=1}^N \gamma_k  \E_{\be^k,\xi^k}\left[\la g(z^k,\be^k,\xi^k), z^k-u\ra \right] 
     &\geq  \sum_{k=1}^{N}\gamma_k  \left(f^\tau\left( \hat x^N,y\right) - f^\tau\left(x, \hat y^N \right)  -  \sqrt{d}M_{2,\delta}\|z^k-u\|_2\right). 
\end{align}
 \end{enumerate}
{\bf Step 3.} (under Assumption \ref{ass:err_noise}) \\
Now we combine  \eqref{eq:eqrefeqffxyz1122}   with   \eqref{eq:inmainsqEtwo1} for  \eqref{eq:inmainsqE} and obtain under Assumption \ref{ass:err_noise} the following
     \begin{align}\label{eq:inmainsqEtwo133}
     &\sum_{k=1}^{N}\gamma_k \left(f^\tau\left( \hat x^N,y\right)  - f^\tau\left(x, \hat y^N \right)
     -   {\sqrt{d} \Delta \|z^k-u\|_2}{\tau^{-1}} \right)  \leq  V_{z^1}(u) +  \sum_{k=1}^N \frac{\gamma_k^2}{2} \left(c  dM_2^2 a^2_q + {d^2 \Delta^2}{\tau^{-2}a_q^2}\right).
     \end{align}
Using Lemma \ref{lm:shamirSmoothFun} we obtain
 \begin{align*}
    f^\tau\left( \hat x^N,y\right) -  f^\tau\left(x, \hat y^N \right) 
    &\geq f\left( \hat x^N,y\right) -  f\left(x, \hat y^N \right)  - 2\tau M_2.
 \end{align*}
 Using this we can rewrite \eqref{eq:inmainsqEtwo133} as follows
 \begin{align}\label{eq:exEeqE1}
       f\left( \hat x^N,y\right) -  f\left(x, \hat y^N \right) 
      &\leq 
     \frac{ V_{z^1}(u)}{\sum_{k=1}^{N}\gamma_k}
      +\frac{c dM_2^2  a^2_q+ d^2 \Delta^2\tau^{-2}a_q^2}{\sum_{k=1}^{N}\gamma_k}\sum_{k=1}^N \frac{\gamma_k^2}{2}  \notag \\
      &+{ \sqrt{d} \Delta \max_{k}\|z^k-u\|_2}{\tau^{-1}} +  2\tau M_2.
 \end{align}
  For the r.h.s. of \eqref{eq:exEeqE1}  we use the definition of  the $\omega$-diameter  of $\Z$:\\
$\D \triangleq \max_{z,v\in \Z}\sqrt{2V_{z}(v)}$ and estimate $\|z^k-u\|_2 \leq \D$ for all $z^1,\dots, Z^k$ and all $u\in \Z$.
 Using this for  \eqref{eq:exEeqE1} and   taking the maximum in $(x,y) \in (\X,\Y)$, we obtain
 \begin{align}\label{eq:dghjsklslsl1}
    \max_{y \in \Y} f\left( \hat x^N,y\right) -  \min_{x\in \X} f\left(x, \hat y^N \right)
     &\leq \frac{ \D^2+ (c  d M_2^2 a^2_q +
     d^2 \Delta^2\tau^{-2} a_q^2)
    \sum_{k=1}^{N} {\gamma_k^2}/{2}}{\sum_{k=1}^{N}\gamma_k}
       \notag \\
    &+ {\sqrt{d}\Delta \D}{\tau^{-1}}
     +  2\tau M_2.  
 \end{align}
Taking the expectation of \eqref{eq:dghjsklslsl1} and choosing  learning rate $\gamma_k 
=\frac{\D}{M_{\rm case1}}\sqrt{\frac{2}{N}}$ with $M^2_{\rm case1} \triangleq cdM_2^2 a_q^2 + d^2\Delta^2\tau^{-2}a_q^2$ in  \eqref{eq:dghjsklslsl1} we get
 \begin{align*}
      \E\left[ \max_{y \in \Y} f\left( \hat x^N,y\right) -  \min_{x\in \X} f\left(x, \hat y^N \right)\right]
     &\leq  
     M_{\rm case1}\D\sqrt{{2}/{N}}
     + {\sqrt{d} \Delta\D}{\tau^{-1}} 
     +  2\tau M_2.  
 \end{align*}
 {\bf Step 4.} (under Assumption \ref{ass:Lip_noise})\\
Now we combine \eqref{eq:eqrefeqffxyz22}  with   \eqref{eq:inmainsqEtwo} for  \eqref{eq:inmainsqE} and obtain under Assumption \ref{ass:Lip_noise} 
\begin{align}\label{eq:inmainsqEtwo334}
    \sum_{k=1}^{N}\gamma_k f^\tau\left( \hat x^N,y\right)  - f^\tau\left(x, \hat y^N \right) 
    &- \sum_{k=1}^N \gamma_k \sqrt{d}M_{2,\delta}\|z^k-u\|_2 \leq V_{z^1}(u)+  \sum_{k=1}^N \frac{\gamma_k^2}{2} c a^2_q d (M_2^2+M_{2,\delta}^2).
\end{align}
Using Lemma \ref{lm:shamirSmoothFun} we obtain
 \begin{align*}
    f^\tau\left( \hat x^N,y\right) -  f^\tau\left(x, \hat y^N \right) 
    &\geq f\left( \hat x^N,y\right) -  f\left(x, \hat y^N \right)  - 2\tau M_2.
 \end{align*}
 Using this we  can rewrite \eqref{eq:inmainsqEtwo334} as follows
 \begin{align}\label{eq:exEeqE}
      f\left( \hat x^N,y\right) -  f\left(x, \hat y^N \right)  
      &\leq 
     \frac{ V_{z^1}(u)}{\sum_{k=1}^{N}\gamma_k}
      +\frac{c  d (M_2^2+M_{2,\delta}^2)a^2_q}{\sum_{k=1}^{N}\gamma_k}\sum_{k=1}^N \frac{\gamma_k^2}{2} \notag\\
      &+ \sqrt{d}M_{2,\delta} \max_{k} \|z^k-u\|_2  +  2\tau M_2. 
 \end{align}
   For the r.h.s. of \eqref{eq:exEeqE}  we use the definition of  the $\omega$-diameter  of $\Z$: \\
$\D \triangleq \max_{z,v\in \Z}\sqrt{2V_{z}(v)}$ and estimate $\|z^k-u\|_2 \leq \D$ for all $z^1,\dots, Z^k$ and all $u\in \Z$.
 Using this for  \eqref{eq:exEeqE} and   taking the maximum in $(x,y) \in (\X,\Y)$, we obtain
 \begin{align}\label{eq:dghjsklslsl}
      \max_{y \in \Y} f\left( \hat x^N,y\right) -  \min_{x\in \X} f\left(x, \hat y^N \right)
     &\leq \frac{ \D^2+ c  d (M_2^2 +M_{2,\delta}^2)a^2_q\sum_{k=1}^N {\gamma_k^2}/{2}}{\sum_{k=1}^{N}\gamma_k}
    \notag\\
     &+  \sqrt{d} M_{2,\delta} \D
     +  2\tau M_2.  
 \end{align}
 Taking the expectation of \eqref{eq:dghjsklslsl1} and choosing  learning rate  $\gamma_k 
=\frac{\D}{M_{\rm case2}}\sqrt{\frac{2}{N}}$ with $M_{\rm case2}^2 \triangleq c d (M_2^2+M_{2,\delta}^2)a_q^2$ in  \eqref{eq:dghjsklslsl} we get
 \begin{align*}
      \E\left[ \max_{y \in \Y} f\left( \hat x^N,y\right) -  \min_{x\in \X} f\left(x, \hat y^N \right)\right]
     &\leq  M_{\rm case2}\D\sqrt{{2}/{N}}  + \sqrt{d}  M_{2,\delta} \D 
     +  2\tau M_2.  
 \end{align*}
\qed

\section{Sketch of the proof of Theorem \ref{th:conv_rate_restarts}}\label{app:proof_restarts}

\textit{Sketch of the proof of Theorem \ref{th:conv_rate_restarts}. }
We repeat the proof of Theorem \ref{cor:bound_noise}, except that now $z^1$ can be chosen in a stochastic way. Moreover, now we use a rougher inequality instead of \eqref{eq:sectermErhs}
\begin{align}\label{eq:sectermErhs_restarts}
  \E_{\be^k} \left[    \left|\la \be^k, z^k-u\ra\right|  \right] \leq  {\D}/{\sqrt{d}}.
\end{align}
{\bf Step 1.} (under Assumption \ref{ass:err_noise}) \\
Taking the expectation in \eqref{eq:exEeqE1}, choosing $(x,y) = (x^{\star}, y^{\star})$, and
 learning rate $\gamma_k 
=\frac{\sqrt{\E [V_{z^1}(z^{\star})]}}{M_{\rm case1}}\sqrt{\frac{2}{N}}$ with $M^2_{\rm case1} \triangleq cdM_2^2 a_q^2 + d^2\Delta^2\tau^{-2}a_q^2$  we get
 \begin{align} \label{rate:restarts1}
      \E\left[ f\left( \hat x^N,y^{\star}\right) -   f\left(x^{\star}, \hat y^N \right)\right]
     &\leq  
     \sqrt{\frac{2}{N}}M_{\rm case1}\sqrt{\E [V_{z^1}(z^{\star})]}
     + {\sqrt{d}\Delta\D}{\tau^{-1}} 
     +  2\tau M_2.  
 \end{align}
 {\bf Step 2.} (under Assumption \ref{ass:Lip_noise})\\
Taking the expectation in \eqref{eq:exEeqE}, choosing $(x,y) = (x^{\star}, y^{\star})$, and  learning rate 
 $\gamma_k 
=\frac{\sqrt{\E [V_{z^1}(z^{\star})]}}{M_{\rm case2}}\sqrt{\frac{2}{N}}$ with $M_{\rm case2}^2 \triangleq c d (M_2^2+M_{2,\delta}^2)a_q^2$  we obtain
 \begin{align} \label{rate:restarts2}
      \E\left[ f\left( \hat x^N,y^{\star}\right) -   f\left(x^{\star}, \hat y^N \right)\right]
     &\leq  \sqrt{\frac{2}{N}}M_{\rm case2}\sqrt{\E [V_{z^1}(z^{\star})]}  + \sqrt{d} M_{2,\delta} \D  
     +  2\tau M_2.  
 \end{align}
{\bf Step 3.} (Restarts)\\
Now let $\tau$ be chosen as $\tau = \O\left(\e/M_2\right)$, where $\e$ is the desired accuracy to solve problem \eqref{eq:min_max_problem}. If one of the two following statement holds
\begin{enumerate}
    \item Assumption \ref{ass:err_noise} holds  and $\Delta =\O\left( \frac{\e^2}{\D M_2\sqrt{ d}}\right)$
    \item Assumption \ref{ass:Lip_noise} holds and 
 $M_{2,\delta} =\O\left( \frac{\e}{\D \sqrt{ d}}\right) $
\end{enumerate}
then  we obtain the convergence rate of the following form
 \begin{align}\label{restart:rate}
     \E\left[ f\left( \hat x^{N_1},y^{\star}\right) -   f\left(x^{\star}, \hat y^{N_1} \right)\right]
     &= \widetilde \O\left(\frac{\sqrt{d} M_2 a_q  }{\sqrt{N_1}}\sqrt{\E [V_{z^1}(z^{\star})]}\right).
 \end{align}
  In this step we will employ the restart technique that is a generalization of the technique proposed in \cite{juditsky2014deterministic}.\\
 For the l.h.s. of   \eqref{restart:rate} we use the Assumption \ref{ass:strongly_convex}. For the r.h.s. of   \eqref{restart:rate} we use the fact $V_{z^1}(z^{\star}) = \widetilde \O (\|z^1 -z^{\star}\|^2_p)$  from  \citet[Remark 3]{gasnikov2018universal}  
 \begin{align}\label{restart:norms}
 \frac{\mu_{r}}{2} \E\left[  \|z^{N_1} - z^{\star}\|_p^{r}\right]
 &\leq
     \E\left[ f\left( \hat x^{N_1},y^{\star}\right) -   f\left(x^{\star}, \hat y^{N_1} \right)\right] = \widetilde \O\left(\frac{ \sqrt{d} M_2  a_q}{\sqrt{{N_1}}} \sqrt{\E\left[\|z^1 -z^{\star}\|_p^2\right]}\right).
 \end{align}
 Then for the l.h.s of   \eqref{restart:norms} we  use the Jensen inequality and  get the following
 \begin{align}
 \frac{\mu_{r}}{2} \left({\E\left[\|z^{N_1} - z^{\star}\|_p^2\right]}\right)^{r/2} &\leq   \frac{\mu_{r}}{2} \E\left[  \|z^{N_1} - z^{\star}\|_p^{r}\right] \leq \E\left[ f\left( \hat x^{N_1},y^{\star}\right) -   f\left(x^{\star}, \hat y^{N_1} \right)\right] \notag\\
     &= \widetilde \O\left( \frac{\sqrt{d} M_2 a_q }{\sqrt{{N_1}}} \sqrt{\E\left[\|z^1 -z^{\star}\|_p^2\right]}\right).
 \end{align}
Finally, 
let us introduce $R_k \triangleq \sqrt{\E\left[\|z^{N_k} - z^{\star}\|_p^2\right]}$ and $R_0 \triangleq \sqrt{\E\left[\|z^{1} - z^{\star}\|_p^2\right]}$. Then we
take ${N_1}$ so as to halve the distance to the solution and get
\[ {N_1} = \widetilde \O \left( \frac{ d M_2^2 a_q^2 }{\mu_{r}^2 R_1^{2(r-1)}} \right).\]
 Next, after $N_1$ iterations, we restart the original method and set $z^1 = z^{N_1}$. We determine $N_2$ similarly: we halve the distance $R_1$ to the solution, and so on. 
 Thus, after $k$  restarts, the total number of iterations will be
 \begin{align} \label{restart:N}
     N = N_1 +\dots + N_k = \widetilde \O \left(\frac{ 2^{2(r-1)}d M_2^2 a_q^2  }{\mu_{r}^2 R_0^{2(r-1)}}\left(1+2^{2(r-1)} + \dots + 2^{2(k-1)(r-1)}\right) \right).
 \end{align} 
 Now we need  to determine the number of restarts. To do  so, we fix the desired accuracy  and using the inequality \eqref{restart:norms} we obtain 
 \begin{align}\label{restart:epsilon}
     \E [\hat{\e}_{\rm sad}] =  \widetilde \O \left(\frac{\mu_{r} R_k^{r}}{2}\right)=   \widetilde \O\left( \frac{ a_q M_2 \sqrt{d} }{\sqrt{N_k}}  R_{k-1}\right) = \widetilde \O \left(\frac{\mu_{r}R_0^{r}}{2^{kr}}\right)  \leq \e.
 \end{align} 
 Then to fulfill this condition, one can choose $k =  \log_{2} (\widetilde \O\left( \mu_{r}R_0^{r}/\e\right)) / r$ and using   \eqref{restart:N} we get the total number of iterations
\[ N = \widetilde \O\left(\frac{ 2^{2k(r-1)}d  M_2^2 a_q^2}{\mu_{r}^2 R_0^{2(r-1)}} \right) = \widetilde \O\left(\frac{ d M_2^2 a_q^2}{\mu_{r}^{2/r} \e^{2(r-1)/r}} \right).\]
If in Theorem \ref{th:conv_rate_restarts}, we use a tighter inequality   \eqref{eq:sectermErhs} instead of
\eqref{eq:sectermErhs_restarts} (as in Theorem \ref{cor:bound_noise}), then
the estimations on the $\Delta$ and $M_{2,\delta}$   can be improved.
Choosing $u = (x^*, y^*)$ we can provide exponentially decreasing sequence of $\D^k = \E \|z^k - u\|_2$ in   \eqref{restart:rate} and get 
\begin{enumerate}
    \item under Assumption \ref{ass:err_noise} $\Delta \lesssim \frac{\mu_{r}^{1/r}\e^{2-1/r}}{M_2\sqrt{d}}$
    \item under Assumption \ref{ass:Lip_noise} 
$M_{2,\delta} \lesssim \frac{\mu_{r}^{1/r}\e^{1-1/r}}{\sqrt{d}}$.
\end{enumerate}
\qed

\section{Sketch of the proof of Theorem \ref{th:conv_rate_restarts_prob}}\label{app:proof_high_bound}

\textit{Sketch of the proof of Theorem \ref{th:conv_rate_restarts_prob}. }

\begin{proof}[Sketch of the proof of  Theorem \ref{th:conv_rate_restarts_prob}]
By the definition $z^{k+1}= {\rm Prox}_{z^k}\left(\gamma_k g(z^k,\be^k,\xi^k) \right)$ we get  \cite{ben2001lectures}, for all $u\in \Z$
\begin{align}
   \gamma_k\la g(z^k,\be^k,\xi^k), z^k-u\ra \leq V_{z^k}(u) - V_{z^{k+1}}(u) + \gamma_k^2\|g(z^k,\be^k,\xi^k)\|_2^2/2   \notag
\end{align}
Summing for $k=1,\dots,N$ we obtain, for all $u\in \Z$
\begin{align}\label{eq:inmainsqE_prob}
   \sum_{k=1}^N \gamma_k\la g(z^k,\be^k,\xi^k), z^k-u\ra 
   &\leq V_{z^1}(u)  + \sum_{k=1}^N\gamma_k^2\|g(z^k,\be^k,\xi^k)\|_2^2/2. 
\end{align}
Next we provide the definition of zeroth-order gradient approximation
similarly to  $\eqref{eq:grad_approx}$ but under zero noise 
\begin{equation}\label{eq:grad_approx_f}
    g_f(z,\xi, \be) = \frac{d}{2\tau}\left(f(z+\tau \be,\xi) - f(z-\tau \be,\xi)   \right) \begin{pmatrix} \be_x \\ -\be_y \end{pmatrix},
\end{equation}
where $\tau>0$ is some constant:
\begin{lemma}[Concentration of Lipschitz functions on the Euclidean unit sphere ]\cite[proof of Proposition 2.10 and Corollary 2.6]{ledoux2001concentration}\label{lm:concentr_Lip}
 For any function $g(\be)$ which is $L$-Lipschitz w.r.t. the $\ell_2$-norm, it holds that if $\be$ is uniformly distributed on the Euclidean unit sphere, then
 \begin{equation*}
     \Prob\left( |g(\be) - \E\left[g(\be) \right] | >t \right)\leq 2 \exp\left( -c'dt^2/L^2 \right),
 \end{equation*}
 where $c'$ is some numerical constant.
\end{lemma}
{\bf Step 1.}\\
Now we estimate the second term in the r.h.s of inequality \eqref{eq:inmainsqE_prob} 
 under Assumption \ref{ass:err_noise}
\begin{align}\label{eq:prob_var}
    \|g(z^k,\be^k,\xi^k)\|_2^2  
    &= \frac{d^2}{4\tau^2}  \left(f(z+\tau \be,\xi) +\delta(z+\tau\be) - f(z-\tau \be,\xi) -\delta(z-\tau\be)  \right)^2  \notag\\
     &\leq  \|g_f(z^k,\be^k,\xi^k)\|_2^2 +   \frac{d^2\Delta^2}{2\tau^2}. \qquad /* \forall a,b, (a-b)^2\leq 2a^2+2b^2 */
\end{align}
To estimate the first term in the r.h.s of   \eqref{eq:prob_var}, we notice that  function $g_f(z^k,\be^k,\xi^k)$    is uniformly $d M_2$-Lipschitz continious and $\E_{\be^k} g_f(z^k,\be^k,\xi^k) = 0$. Thus, we use Lemma \ref{lm:concentr_Lip}
 and obtain for some constant $c$: 
\begin{align}\label{eq:estim_of_g_f}
    \Prob \left( \| g_f(z^k,\be^k,\xi^k) \|_2 > t  \right) \leq 2\exp \left( - \frac{c t^2}{M_2^2 d} \right) 
\end{align}
Let us denote 
$\phi_k = \gamma_k^2 \|g_f(z^k,\be^k,\xi^k)\|_2^2/2$ and
$\sigma_k = 2 \gamma_k^2 d M_2^2 $. Then we consider conditional expectation 
\begin{align*}
     \E_{|k-1} \left[ \exp \left(\frac{|\phi_k|}{\sigma_k}\right)\right] & = \E_{|t-1} \left[ \exp\left( \frac{\|g_f(z^k,\be^k,\xi^k)\|_2^2}{4 d M_2^2 }\right) \right]  \\
     &= \int_0^{\infty} \Prob \left( \exp \left( \frac{\|g_f(z^k,\be^k,\xi^k)\|_2^2}{4 d  M_2^2 }\right) \geq \tilde{t}  \right) d \tilde{t} \\
     &\leq \int_0^{1} 1 d \tilde{t} 
      + \int_1^{\infty} \underbrace{\Prob \left( \exp \left( \frac{\|g_f(z^k,\be^k,\xi^k)\|_2^2}{4 d M_2^2 }\right) \geq \tilde{t}  \right)}_{\text{by } \eqref{eq:estim_of_g_f} : ~\leq \exp \left( - \frac{c \sigma_k^2  \ln \tilde{t}}{ d M_2^2 } \right) \leq \left(\frac{1}{\tilde{t}} \right)^{\frac{c \sigma_k^2}{d M_2^2}} = \left(\frac{1}{\tilde{t}} \right)^{2}} d \tilde{t} \leq 1 + 1 \notag \\
      &\leq  \exp(1).
\end{align*}
Thus, we can use \cite[Lemma 2, case B]{lan2012validation} and get
\begin{align*}
    &\Prob \left( \left| \sum_{k=1}^N \gamma_k^2 \|g_f(z^k,\be^k,\xi^k)\|_2^2  \right| > 2 d M_2^2  \sum_{k=1}^N \gamma_k^2 + 2\Omega   d M_2^2  \sqrt{\sum_{k=1}^N \gamma_k^4} \right) \notag \\
    &\leq \exp \left(-\frac{\Omega^2}{12}\right) + \exp\left(- \frac{3\sqrt{N} \Omega}{4}\right) 
\end{align*}

Thus using  this and  \eqref{eq:prob_var}  we can estimate the second term in the r.h.s of inequality \eqref{eq:inmainsqE_prob} as follows
\begin{align}
    &\Prob \left( \left| \sum_{k=1}^N \gamma_k^2 \|g(z^k,\be^k,\xi^k)\|_2^2/2 \right| > 2  d M_2^2  \sum_{k=1}^N \gamma_k^2 + 2\Omega d  M_2^2  \sqrt{\sum_{k=1}^N \gamma_k^4} +  \frac{d^2\Delta^2 }{2 \tau^2  }\sum_{k=1}^N \gamma_k^2 \right)  \notag \\
    &\leq \exp \left(-\frac{\Omega^2}{12} \right) + \exp \left(-\frac{3\sqrt{N} \Omega}{4} \right) \label{eq:first_step_second_term}
\end{align}
{\bf Step 2.}\\
Using the notation \eqref{eq:grad_approx_f} we rewrite the l.h.s. of    \eqref{eq:inmainsqE_prob} as following:
\begin{align}
    \sum_{k=1}^N \gamma_k\la g(z^k,\be^k,\xi^k), z^k-u\ra &=\sum_{k=1}^N \gamma_k\la g_f(z^k,\be^k,\xi^k) - \nabla f^{\tau} (z^k) , z^k-u\ra + \nonumber\\
    & + \sum_{k=1}^N \gamma_k\la \frac{d}{2\tau} (\delta(z^k+\tau \be) - \delta(z^k-\tau \be)) \be   , z^k-u\ra \notag \\
    &+ \sum_{k=1}^N \gamma_k\la \nabla f^{\tau} (z^k) , z^k-u\ra. \label{eq:step2_rewriten}
\end{align}
\begin{enumerate}[leftmargin=*]
    \item For the first term in the  r.h.s. of    \eqref{eq:step2_rewriten}, we provide the following notions:
$$\sigma_k = 2 \gamma_k \sqrt{d} M_2  \sqrt{V_{z^1}(u)} \Omega ,$$
$$\phi_k (\xi_k)= \gamma_k\la g_f(z^k,\be^k,\xi^k) - \nabla f^{\tau} (z) , z^k-u\ra .$$
For applying the case A of Lemma 2 from \cite{lan2012validation} we need to estimate the module of function $\gamma_k\la g_f(z^k,\be^k,\xi^k) , z^k-u\ra$. Using Assumption \ref{ass:uni_Lip_obj} we obtain:
\begin{align}
    \left| \gamma_k\la g_f(z^k,\be^k,\xi^k) , z^k-u\ra \right| &\stackrel{\eqref{eq:grad_approx_f}}{=}  \left| \gamma_k\la \frac{d}{2\tau}\left(f(z+\tau \be,\xi) - f(z-\tau \be,\xi)   \right) \begin{pmatrix} \be_x \\ -\be_y \end{pmatrix} , z^k-u\ra \right|\nonumber\\
    & \stackrel{Ass. \ref{ass:uni_Lip_obj}}{\leq} \gamma_k d M_2 \left| \left\la  \be , z^k-u\right\ra \right|. \label{eq:estim_of_scalar_product}
\end{align}
Now we need to estimate the term $\left| \la  \be , z^k-u\ra \right|$. To do this, using the Poincaré's lemma from \cite{lifshits2012lectures} paragraph 6.3 we rewrite $\be$ in different form:
\begin{align}\label{eq:poincare_def_e}
    \be \stackrel{D}{=} \frac{\eta}{\sqrt{\eta_1^2 + \dots + \eta_d^2}},
\end{align}

where $\eta = \left( \eta_1, \eta_2, \dots , \eta_d \right)^T = \mathcal{N} \left( 0 , I_d \right)$. 
Using Lemma 1 \cite{laurent2000adaptive}
it follows:
\begin{align}
    \Prob \left(  \sum_{k=1}^d \left|\eta_k\right|^2 \le d -  2 \sqrt{\Omega d} \right) \leq  \exp(- \Omega). \label{eq:step2_sphere_rewriten}
\end{align}
Using the definition of $\eta$ we obtain:
\begin{align}
    \la \eta, z^k-u \ra \stackrel{D}{=} \mathcal{N} \left( 0 , \|z^k - u\|_2^2 \right). \label{eq:step2_poincare}
\end{align}
Using   \eqref{eq:step2_sphere_rewriten}, \eqref{eq:step2_poincare} and \eqref{eq:poincare_def_e} we can estimate the r.h.s of   \eqref{eq:estim_of_scalar_product}:
\begin{align}
    \Prob &\left(     \left| \la \be, z^k-u \ra \right| > \frac{ \Omega V_{z^k}(u)}{\sqrt{d}} \right) \leq 3 \exp (-\Omega/4). \label{eq:step1_sphere_rewriten_2}
\end{align}

Moreover, using   \eqref{eq:estim_of_scalar_product},   \eqref{eq:step1_sphere_rewriten_2} and the fact that $V_{z^k}(u) = \O ( V_{z^1}(u))$ for all $k \geq 1$ from \cite{gorbunov2021gradient} we have:
\begin{align*}
    \Prob \left(\exp \left( \left(\gamma_k\la g_f(z^k,\be^k,\xi^k) - \nabla f^{\tau} (z) , z^k-u\ra \right)^2/ \sigma_k^2 \right)  \geq \exp(1)\right) \leq 3 \exp (-\Omega/4).
\end{align*}
 From \cite[case A of Lemma 2]{lan2012validation} it holds
\begin{align}\label{eq:step2_martingale_rewriten}
    &\Prob \left( \left| \sum_{k=1}^N \gamma_k\la g_f(z^k,\be^k,\xi^k) - \nabla f^{\tau} (z) , z^k-u\ra  \right| > 4 \sqrt{d} \Omega^2 \sqrt{V_{z^1}(u)} M_2  \sqrt{  \sum_{k=1}^N \gamma_k^2 } \right) \notag \\
    &\leq  2 \exp (-\Omega^2/3) + 3 \exp (-\Omega/4). 
\end{align}
\item For the second term in r.h.s of   \eqref{eq:step2_rewriten} under the Assumption \ref{ass:err_noise} we obtain :
\begin{align}
    & \sum_{k=1}^N \gamma_k\la \frac{d}{2\tau} (\delta(z^k+\tau \be) - \delta(z^k-\tau \be)) \be   , z^k-u\ra \geq - \sum_{k=1}^N \gamma_k \frac{d \Delta}{\tau} \left|\la \be, z^k-u \ra \right| \label{eq:step2_delta_rewriten}
\end{align}

Using   \eqref{eq:step1_sphere_rewriten_2} we can estimate the r.h.s of   \eqref{eq:step2_delta_rewriten}:
\begin{align}
    \Prob &\left(   \sum_{k=1}^N \gamma_k \frac{d \Delta}{\tau} \left| \la \be, z^k-u \ra \right| > \frac{ \Omega \D \sqrt{d} \Delta }{\tau}\sum_{k=1}^N \gamma_k \right) \leq 3 \exp(-\Omega/4). \label{eq:step2_sphere_rewriten_2}
\end{align}
\item We rewrite the third term in r.h.s. of    \eqref{eq:step2_rewriten} in the following form: 
\begin{align} \label{eq:eqexpfxy_prob}
    \sum_{k=1}^N \gamma_k\la \nabla f^{\tau} (z^k) , z^k-u\ra 
     &=  \sum_{k=1}^N \gamma_k\left( \la \nabla_x f^\tau(x^k,y^k), x^k-x \ra  -  \la \nabla_y  f^\tau(x^k,y^k), y^k-y \ra \right) \notag \\
     &\geq  \sum_{k=1}^N \gamma_k ( f^\tau(x^k,y^k) - f^\tau(x,y^k))  -  ( f^\tau(x^k,y^k) - f^\tau(x^k,y)) \notag \\
     &=  \sum_{k=1}^N \gamma_k (f^\tau(x^k,y)  - f^\tau(x,y^k)).
\end{align}
Then we use the fact function $f^\tau(x,y)$ is convex in $x$ and concave in $y$ and  obtain 
\begin{align}\label{eq:fineqE_prob}
     \frac{1}{\sum_{i=1}^{N}\gamma_k} \sum_{k=1}^N \gamma_k (f^\tau(x^k,y)  - f^\tau(x,y^k))  &\leq   f^\tau\left(\hat x^N,y\right)  - f^\tau\left(x, \hat y^N \right) ,
\end{align}
where $ (\hat x^N, \hat y^N)$ is the output of the Algorithm \ref{alg:zoSPA}. 
Using   \eqref{eq:fineqE_prob}  for   \eqref{eq:eqexpfxy_prob} we get 
\begin{align}\label{eq:fpslff_prob}
      \sum_{k=1}^N \gamma_k\la \nabla f^\tau(z^k), z^k-u\ra  
     &\geq  \sum_{k=1}^{N}\gamma_k  \left( f^\tau\left( \hat x^N,y\right)  - f^\tau\left(x, \hat y^N \right) \right) .
\end{align}
\end{enumerate}

{\bf Step 4.} \\
Then we use this with    \eqref{eq:step2_rewriten}, \eqref{eq:step2_delta_rewriten} and \eqref{eq:fpslff_prob} for   \eqref{eq:inmainsqE_prob} and obtain
\begin{align}
      &f^\tau\left( \hat x^N,y\right)   - f^\tau\left(x, \hat y^N \right)  \leq   \frac{V_{z^1}(u)}{\sum_{k=1}^{N}\gamma_k } + \frac{1}{\sum_{k=1}^{N}\gamma_k}  \sum_{k=1}^N \frac{\gamma_k^2}{2} \|g(z^k,\be^k,\xi^k)\|_2^2  \nonumber \\
    & -\frac{1}{\sum_{k=1}^{N}\gamma_k}  \sum_{k=1}^N \gamma_k\la g_f(z^k,\be^k,\xi^k) - \nabla f^{\tau} (z^k) , z^k-u\ra + \frac{1}{\sum_{k=1}^{N}\gamma_k}  \sum_{k=1}^N \gamma_k \frac{d \Delta}{\tau} \la \be, z^k-u \ra. \label{eq:exEeqE_prob}
\end{align}

 For the l.h.s. of   \eqref{eq:exEeqE_prob} we use Lemma \ref{lm:shamirSmoothFun} and obtain
 \begin{align}
    f^\tau\left( \hat x^N,y\right) -  f^\tau\left(x, \hat y^N \right) 
    &\geq f\left( \hat x^N,y\right) -  f\left(x, \hat y^N \right)  - 2\tau M_2. \label{eq:f_tau_to_f}
 \end{align}
 Using   \eqref{eq:first_step_second_term}, \eqref{eq:step2_martingale_rewriten}, \eqref{eq:step2_sphere_rewriten_2}, \eqref{eq:f_tau_to_f} and taking the maximum in $(x,y) \in (\X,\Y)$, we obtain we can rewrite   \eqref{eq:exEeqE_prob} as follow:
 
 \begin{align}\label{eq:dghjsklslsl_proba}
       \Prob & \left( \max_{y \in \Y} f\left( \hat x^N,y\right) -  \min_{x\in \X} f\left(x, \hat y^N \right) \geq \frac{4 \Omega^2 \sqrt{d} \sqrt{V_{z^1}(u)} M_2  \sqrt{  \sum_{k=1}^N \gamma_k^2 }}{{\sum_{k=1}^{N}\gamma_k}} + \frac{2  M_2^2 d \sum_{k=1}^N \gamma_k^2}{c \sum_{k=1}^N \gamma_k }  \right. \nonumber\\
      & \left. \frac{ \Omega \D \sqrt{d} \Delta }{\tau } + \Omega \frac{2  M_2^2 d \sqrt{\sum_{k=1}^N \gamma_k^4}}{c \sum_{k=1}^N \gamma_k } +  \frac{d^2\Delta^2 \sum_{k=1}^N \gamma_k^2}{2 \tau^2 \sum_{k=1}^N \gamma_k } + \frac{V_{z^1}(u)}{\sum_{k=1}^{N}\gamma_k }  +2 \tau M_2 \right)\notag\\
      & \hspace{2cm}\leq \exp (-\Omega^2/12) + \exp (-3\sqrt{N} \Omega/4) +  2 \exp (-\Omega^2/3) + 6 \exp (-\Omega/4).
\end{align}

 Choosing the stepsize $\gamma_k 
=\frac{\sqrt{V_{z^1}(u)}}{M}\sqrt{\frac{2}{N}}$ with $M^2 \triangleq d M_2^2+d^2 \Delta^2 \tau^{-2}$ in   \eqref{eq:dghjsklslsl_proba} we obtain
\begin{align}
       \Prob & \left( \max_{y \in \Y} f\left( \hat x^N,y\right) -  \min_{x\in \X} f\left(x, \hat y^N \right) \geq \frac{4 \Omega^2 \sqrt{d} \sqrt{V_{z^1}(u)} M_2  }{\sqrt{N}} + \frac{2 \sqrt{2}  \sqrt{V_{z^1}(u)} M_2^2 d}{c M\sqrt{N}}   + \right. \nonumber\\
      & \hspace{2cm}\left. \frac{ \Omega \D \sqrt{d} \Delta }{\tau } + \Omega \frac{2 \sqrt{2} M_2^2 d \sqrt{V_{z^1}(u)}}{c M N} +  \frac{d^2\Delta^2 \sqrt{2} \sqrt{V_{z^1}(u)}}{2 \tau^2 M \sqrt{N} }+ \frac{\sqrt{V_{z^1}(u)} M}{\sqrt{2N} }  +2 \tau M_2 \right) \nonumber\\
      & \hspace{2cm} \leq \exp (-\Omega^2/12) + \exp(-3\sqrt{N} \Omega/4) + 2 \exp(-\Omega^2/3) + 6 \exp (-\Omega/4).
\end{align}

Using the notation of $M$ we obtain:
\begin{align}\label{eq:conv_proba}
       &\Prob  \left( \max_{y \in \Y} f\left( \hat x^N,y\right) -  \min_{x\in \X} f\left(x, \hat y^N \right) \geq \left(4 \Omega^2 + \frac{2 \sqrt{2}}{c} + \frac{2}{\sqrt{2}}\right) \frac{  \sqrt{V_{z^1}(u)} M  }{\sqrt{N}}   + \frac{ \Omega \D \sqrt{d} \Delta }{\tau }  +2 \tau M_2 + \right. \nonumber\\
      & \left.+ \Omega \frac{2 \sqrt{2} M \sqrt{V_{z^1}(u)}}{c N}    \right)  \leq \exp (-\Omega^2/12) + \exp (-3\sqrt{N} \Omega/4) + 2 \exp(-\Omega^2/3) + 6 \exp (-\Omega/4).
\end{align}

Now we need to get a more compact convergence result in the form $\Prob\left\{\hat{\e}_{\rm sad}(\hat z^N) \leq \e \right\} \geq 1 - \sigma$. For this we fix $\sigma > 0$ and let $\tau$ be chosen as $\tau = \O\left(\e/M_2\right)$. If moreover Assumption \ref{ass:err_noise} holds true with $\Delta =\O\left( \frac{\e^2}{\D M_2\sqrt{ d}}\right)$, then taking $(x, y) = (x^{\star}, y^{\star})$ in \eqref{eq:conv_proba} we can choose $\Omega$ small enough to obtain:
 \begin{align}\label{restart:rate_prob}
     \Prob & \left( f\left( \hat x^{N},y^{\star}\right) -  f\left(x^{\star}, \hat y^{N} \right) = \widetilde \O\left(\ M_2 \sqrt{d}\frac{\|z^1 -z^{\star}\|_p }{\sqrt{N}}  \right)   \right)  \geq 1 - \sigma.
 \end{align}
 We note that now the notation $\widetilde \O$ contains the factor $\log \sigma^{-1}$.

 {\bf Step 3.} (Restarts)\\
In this step we will employ the restart technique that is generalization of technique proposed in \cite{juditsky2014deterministic}.
  For the l.h.s. of   \eqref{restart:rate_prob} we use Assumption \ref{ass:strongly_convex} then with probability at least $1 - \sigma_1$ 
 \begin{align}\label{restart:norms_prob}
 \frac{\mu_{r}}{2} \underbrace{\|z^{N_1} - z^{\star}\|_p^{r}}_{R_1^r}
 &\leq
     f\left( \hat x^{N_1},y^{\star}\right) -   f\left(x^{\star}, \hat y^{N_1} \right) = \widetilde \O\left(\ M_2 \sqrt{d}\frac{\|z^1 -z^{\star}\|_p }{\sqrt{N_1}} \right).
 \end{align}
 Then taking ${N_1}$ so as to reduce the distance to the solution by half, we obtain
 $$ {N_1} = \widetilde \O_1 \left( \frac{M_2^2 d}{\mu_{r}^2 R_1^{2(r-1)}} \right).$$
 Next, after $N_1$ iterations, we restart the original method and set $z^1 = z^{N_1}$. We determine $N_2$ from a similar condition for reducing the distance $R_1$ to the solution by a factor of $2$, and so on.  
 
 We remind that at each restart step $i \in \{1,k\}$, the resulting formula \eqref{restart:norms_prob} is valid only with probability $1 - \sigma_i$. Thus, we choose $\sigma_i = \sigma/k$ and then by the union bound inequality all inequalities are satisfied simultaneously with probability at least $1 - \sigma$. We will determine the number of restarts further, but at this stage we use the fact that $k$ depends on the accuracy only logarithmically, which entails that notations $\widetilde \O_i$ and $\widetilde \O$ are equivalent $\forall i\in \{1,k\}$. Thus, after $k$ of such restarts, the total number of iterations will be
 \begin{align} \label{restart:N_prob}
     N = N_1 +\dots + N_k = \widetilde \O \left(\frac{ 2^{2(r-1)}d M_2^2 }{\mu_{r}^2 R_0^{2(r-1)}}\left(1+2^{2(r-1)} + \dots + 2^{2(k-1)(r-1)}\right) \right).
 \end{align} 
 It remains for us to determine the number of restarts, for this we fix the desired accuracy in terms of $\Prob\left\{\hat{\e}_{\rm sad}(\hat z^N) \leq \e \right\} \geq 1 - \sigma$ and using the inequality \eqref{restart:norms_prob} we obtain 
 \begin{align}\label{restart:epsilon_prob}
     \hat{\e}_{\rm sad} = \widetilde \O \left(\frac{\mu_{r} R_k^{r}}{2}\right)=   \frac{\sqrt{d} M_2  \widetilde \O\left(R_{k-1}\right)}{\sqrt{N_k}}  = \widetilde \O \left(\frac{\mu_{r}R_0^{r}}{2^{kr}}\right) \leq \e.
 \end{align} 
 Then to fulfill this condition one can choose $k =  \log_{2} (\widetilde \O\left( \mu_{r}R_0^{r}/\e\right)) / r$ and using   \eqref{restart:N} we get the total number of iterations
\[ N = \widetilde \O\left(\frac{2^{2k(r-1)}d M_2^2 }{\mu_{r}^2 R_0^{2(r-1)}} \right) = \widetilde \O\left(\frac{dM_2^2 }{\mu_{r}^{2/r} \e^{2(r-1)/r}} \right).\]

\end{proof}

\end{document}